\documentclass[10pt]{article}
\usepackage{amsmath,amsthm,amssymb,amsfonts,amscd}
\usepackage{xcolor}
\newtheorem{theorem}{Theorem}[section]
\newtheorem{definition}[theorem]{Definition}
\newtheorem{lemma}[theorem]{Lemma}
\newtheorem{corollary}[theorem]{Corollary}

\newtheorem{remark}[theorem]{Remark}

\font\fivemsb=msbm5
\font\sevenmsb=msbm7
\font\tenmsb=msbm10
\newfam\msbfam
\textfont\msbfam=\tenmsb
\scriptfont\msbfam\sevenmsb
\scriptscriptfont\msbfam\fivemsb

\begin{document}


\title
{\textbf{A generalized analytic Fourier-Feynman transform 
on the product function space $C_{a,b}^2[0,T]$ and related topics}}


\author{Jae Gil Choi\\
Department of Mathematics, Dankook University\\
Cheonan 330-714,  Korea\\
jgchoi@dankook.ac.kr
\and
David Skoug\\
Department of Mathematics, University of Nebraska-Lincoln\\ 
Lincoln, NE 68588-0130,   USA\\
dskoug@math.unl.edu
\and
Seung Jun Chang\\
Department of Mathematics,  Dankook University\\
Cheonan 330-714, Korea\\
sejchang@dankook.ac.kr}

\date{\empty}

\maketitle

\begin*

\thanks{{\bf Abstract}:  
In this paper 
we obtain  various results involving
the generalized analytic Fourier-Feynman transform 
and the first variation of functionals 
in a  Fresnel type class 
defined on the product function space $C_{a,b}^2[0,T]$.}\\

\addtolength{\columnsep}{15mm}

\thanks{{\bf Keywords}: 
Generalized Brownian motion,
Cameron-Martin like subspace,
generalized analytic Feynman integral,
generalized analytic Fourier-Feynman transform,
first variation,
generalized Fresnel type class,
change of scale formula,
translation theorem,
Cameron-Storvick type theorem.}\\

\thanks{{\bf Mathematics Subject Classification}: 
Primary 60J65, 28C20}\\

\end*

\maketitle

\tableofcontents

\setcounter{equation}{0}
\section{Introduction}\label{sec:1}

Let $H$ be a separable Hilbert space 
and let $\mathcal{M}(H)$ be the space of 
all complex-valued Borel measures on $H$.
The Fourier transform of $\sigma$ in $\mathcal{M}(H)$
is defined by
\begin{equation}\label{eq:intro001}
f(\sigma)(h')
   \equiv   \widehat{\sigma}(h')
   =\int_H \exp\{i\langle{h,h'}\rangle\}d\sigma(h), 
   \quad h'\in H.
\end{equation}
The set of all functions of the form \eqref{eq:intro001} 
is denoted by  $\mathcal{F}(H)$ 
and is called the Fresnel class of $H$.
Let $(H,B,\nu)$ be an abstract Wiener space.
It is known   \cite{KB84, KKK85} that
each functional of the form \eqref{eq:intro001}
can be extended to $B$ uniquely by
\begin{equation}\label{eq:intro002}
\widehat{\sigma} (x)
 =\int_H \exp\{i(h,x)^{\sim} \}d\sigma(h), 
 \quad x\in B
\end{equation}
where $(\cdot,\cdot)^{\sim}$ is a stochastic inner product 
between  $H$ and $B$. 
The Fresnel class $\mathcal{F}(B)$ of $B$ 
is the space of (equivalence classes of) 
all functionals of the form \eqref{eq:intro002}.
There has been a tremendous amount of papers and books 
in the literature on the Fresnel integral theory
and Fresnel classes  $\mathcal{F}(B)$ and $\mathcal{F}(H)$ 
on  abstract Wiener  and Hilbert spaces.
For an elementary introduction, see   \cite[Chapter 20]{JL00}.

\par
Furthermore, in \cite{KB84}, 
Kallianpur and Bromley introduced a larger class 
$\mathcal{F}_{A_1,A_2}$  than the Fresnel class  $\mathcal{F}(B)$ 
and showed the existence of the analytic Feynman integral 
of functionals in  $\mathcal{F}_{A_1,A_2}$  
for a successful treatment of certain physical problems
by means of a Feynman integral.
The Fresnel class $\mathcal{F}_{A_1,A_2}$ of $B^2$ 
is the space of (equivalence classes of) all
functionals on $B^2$ of the form 
\[
F(x_1,x_2)
 =\int_{H^2}\exp\bigg\{\sum_{j=1}^2 i(A_j^{1/2}h,x_j)^{\sim} \bigg\}
  d \sigma(h)
\]
where $A_1$ and $A_2$ are bounded, nonnegative and self-adjoint 
operators on $H$ and $\sigma\in \mathcal{M}(H)$.

\par
Let $A$ be a nonnegative self-adjoint operator on $H$ 
and let $\sigma$  be any complex Borel measure on $H$. 
Then the functional 
\begin{equation}\label{1eq:element-FB}
F(x)
 =\int_{H}\exp\{i(A^{1/2}h,x)^{\sim}\} d \sigma(h)
\end{equation}
belongs to the Fresnel class $\mathcal{F} (B)$ on $B$  
because it can be rewritten as
\[
\int_{H}\exp\{i(h,x)^{\sim}\} d\sigma_A(h)
\] 
for $\sigma_A=\sigma \circ(A^{1/2})^{-1}$.
For the functional $F$ given by equation \eqref{1eq:element-FB},
the  analytic Feynman integral 
$\int_B^{\text{\rm anf}_{1}} F(x) d \nu(x)$ with parameter $q=1$
(based on the connection with the Fresnel integral 
of $F$ in $\mathcal{F}(H)$ by Albeverio and H{\o}egh-Krohn,
the most important value  of the parameter $q$ is $q=1$)
on $B$ exists and is given by
\begin{equation}\label{1eq:Fresnel-integral}
\int_B^{\text{\rm anf}_{1}}F(x)d\nu(x)
=\int_H \exp\bigg\{-\frac{i}{2}\langle{Ah,h}\rangle\bigg\}d \sigma(h).
\end{equation}
If we choose  $A$ to  be the identity operator on $H$, 
then equation \eqref{1eq:Fresnel-integral} 
is equal to `the Fresnel  integral  $\mathcal{F}(f)$' 
of $f(\sigma)$,  studied by  Albeverio and H{\o}egh-Krohn in \cite{AH76}.
The concept of  the Fresnel integral is not based on 
the technique of analytic continuation, 
but appears  as solutions of important problems 
in quantum mechanics and in quantum field theory.

\par
Let $A$ be a bounded self-adjoint operator on $H$.
Then we may write 
\begin{equation*}\label{1eq:decomposition}
A=A_+-A_-
\end{equation*}
where $A_+$ and $A_-$ are each bounded, nonnegative and self-adjoint.
Take $A_1=A_+$ and $A_2=A_-$ in the definition of 
$\mathcal{F}_{A_1,A_2}$ above. 
For any $F$ in $ \mathcal{F}_{A_+,A_-}$, 
the analytic Feynman integral  of $F$ with parameter $(1,-1)$ 
is given by
\begin{equation}\label{1eq:op}
\int_{B^2}^{\text{\rm{anf}}_{(1,-1)}} 
  F(x_1,x_2)d(\nu\times\nu)(x_1,x_2)
 =\int_H
    \exp\bigg\{  -\frac{i}{2}\langle{ Ah,h\rangle}\bigg\} d \sigma(h).
\end{equation}
Kallianpur and Bromley, using this idea,  
studied and  investigated  relationships 
between the    Albeverio and H{\o}egh-Krohn's 
Fresnel integral with respect to a symmetric bilinear form 
$\Delta$ on $H$ (see \cite[Chapter 4]{AH76}) 
and the analytic Feynman integral given by equation \eqref{1eq:op}.

\par
In this paper we   extend the ideas in \cite{KB84}
from the functionals on $B^2$
to the functionals on the product function space $C_{a,b}^2[0,T]$.
The function space $C_{a,b}[0,T]$, induced by  
generalized Brownian motion,
was introduced by J. Yeh \cite{Yeh71, Yeh73}
and was used extensively by Chang and Chung \cite{CC96},
Chang and Skoug \cite{CS03}, and Chang, Chung and Skoug \cite{CCS09}.
In this paper we  also construct a concrete theory of 
the generalized analytic Fourier-Feynman 
transform(GFFT) and the first variation 
of functionals in a generalized 
Fresnel type class defined on $C_{a,b}^2[0,T]$.
Other work involving  GFFT theories   on $C_{a,b}[0,T]$ 
include \cite{CCS03, CCS10, CC12}.

\par
The Wiener process  used in \cite{Cameron51,CS80,CS82,CS87I,
CS87II,CS91,CKY00,CSY01,CS02,Johnson82,JS83,KB84,KKK85,YK07}
is stationary in time and is  free of  drift while the stochastic process 
used in this paper, as well as in \cite{CC96,CS03,CCS03,CCS04,CCS07,
CCS10,CCL09,CC12,Yeh71}, is nonstationary in time, is subject to   
a drift $a(t)$, and   can be used to explain  the position of 
the Ornstein-Uhlenbeck process in an external force field \cite{Nelson}.  
It turns out,  as noted in 
Remark \ref{re:effect-a} below, that including a  drift 
term $a(t)$ makes establishing the existence  of the GFFT
of functionals on   $C_{a,b}^2[0,T]$ very difficult. 
However, when  $a(t)\equiv 0$ and $b(t)=t$ 
on $[0,T]$, the general function space  $C_{a,b}[0,T]$ 
reduces to the Wiener space $C_0[0,T]$.

\setcounter{equation}{0}
\section{Definitions and preliminaries}\label{sec:2}

Let $D=[0,T]$ and let $(\Omega, \mathcal{W}, P)$
be a probability measure space. A real-valued 
stochastic process $Y$ on $(\Omega, \mathcal{W}, P)$
and $D$ is called a  generalized Brownian motion  
process if $Y(0,\omega)=0$ almost everywhere and for 
$0=t_1< t_2< \cdots < t_n \le T$, the $n$-dimensional 
random vector $(Y(t_1, \omega), \ldots, Y(t_n, \omega))$
is normally distributed  with density function
\[
\begin{split}
K_n(\vec t,\vec u)
&=\bigg(\prod_{j=1}^n 2
\pi \big(b(t_j)-b(t_{j-1})\big)\bigg)^{-1/2}\\
&\times
\exp\bigg\{-\frac12\sum_{j=1}^n 
\frac{\big[(u_j-a(t_j))
   -(u_{j-1}-a(t_{j-1}))\big]^2}{b(t_j)-b(t_{j-1})}
\bigg\}
\end{split}
\]
where $\vec u=(u_1,\ldots, u_n)$, $u_0=0$, 
$\vec t=(t_1,\ldots, t_n)$,  $a(t)$ is an absolutely 
continuous real-valued  function on $[0,T]$  
with $a(0)=0$, $a'(t) \in L^2[0,T]$,  
and $b(t)$ is a strictly increasing, 
continuously differentiable real-valued function
with $b(0)=0 $ and $b'(t) >0$ for each $t \in [0,T]$.

\par
As explained in \cite[p.18--20]{Yeh73}, $Y$ induces a probability measure 
$\mu$  on the measurable space $(\mathbb R^{D},\mathcal{B}^D)$ where $\mathbb R^D$
is the space of all real-valued functions $x(t)$, $t\in D$,
and $\mathcal{B}^D$ is the smallest $\sigma$-algebra of subsets of 
$\mathbb R^{D}$  with respect to which all the coordinate evaluation maps $e_t (x) = x(t)$ 
defined on $\mathbb R^{D}$ are measurable. The triple $(\mathbb R^{D},\mathcal{B}^D,\mu)$  
is a probability measure space. This measure space is called the function
space induced by the generalized Brownian motion process $Y$ determined by $a(\cdot)$
and $b(\cdot)$.

\par
In  \cite{Yeh73}, Yeh shows that the generalized 
Brownian motion process $Y$ determined by $a(\cdot)$ 
and $b(\cdot)$ is a Gaussian process with mean function 
$a(t)$ and covariance function $r(s,t)=\min\{b(s),b(t)\}$,
and that the probability measure $\mu$ induced by $Y$, 
taking a separable version, is supported by 
$C_{a,b}[0,T]$ (which is equivalent to the Banach space of 
continuous functions $x$ on $[0,T]$ with $x(0)=0$  
under the sup norm).  Hence,  $(C_{a,b}[0,T],\mathcal 
B(C_{a,b}[0,T]),\mu )$ is the function space induced by $Y$
where $\mathcal B(C_{a,b}[0,T])$ is the Borel 
$\sigma$-algebra of $C_{a,b}[0,T]$.
We then complete this function space
to obtain $(C_{a,b}[0,T],\mathcal W(C_{a,b}[0,T]),\mu )$ 
where $\mathcal W(C_{a,b}[0,T])$ is the set of all 
Wiener measurable subsets of $C_{a,b}[0,T]$.

\par
A subset $B$ of $C_{a,b}[0,T]$ is said to be 
scale-invariant measurable    
provided $\rho B$ is $\mathcal{W}(C_{a,b}[0,T])$-measurable for all $\rho>0$, 
and a scale-invariant measurable set $N$ 
is said to be a scale-invariant null set 
provided $\mu(\rho N)=0$ for all $\rho>0$. 
A property that holds except on a scale-invariant null set 
is  said to hold scale-invariant almost everywhere(s-a.e.).
A functional $F$ is said to be scale-invariant measurable 
provided $F$ is defined on a scale-invariant measurable set 
and $F(\rho\,\,\cdot\,)$ 
is $\mathcal{W}(C_{a,b}[0,T])$-measurable for every $\rho>0$.
If two functionals $F$ and $G$ defined on $C_{a,b}[0,T]$ 
are equal s-a.e.,   we write $F\approx G$.

\par
Let $L_{a,b}^2[0,T]$ be the set of functions on $[0,T]$ 
which are Lebesgue measurable and square integrable 
with respect to the Lebesgue-Stieltjes measures on $[0,T]$ 
induced by  $a(\cdot)$ and $b(\cdot)$; 
i.e., 
\[
L_{a,b}^2[0,T] 
=\bigg\{   v :  \int_{0}^{T} v^2 (s) db(s)  <\infty  \hbox{ and }
                \int_0^T v^2 (s) d |a|(s) < \infty \bigg\}         
\]
where $|a|(\cdot)$ is the   total variation   function of $a(\cdot)$.
Then $L_{a,b}^2[0,T]$ is a separable Hilbert space
with inner product defined by
\[
(u,v)_{a,b}=\int_0^T u(t)v(t)d[ b(t) +|a| (t)].               
\]
In particular, 
note that $\| u\|_{a,b}\equiv\sqrt{(u,u)_{a,b}} =0$ 
if and only if $u(t)=0$ a.e. on $[0,T]$.

\par
Let $\{ \phi_{j}\} _{j=1}^{\infty}$ be a complete orthonormal  set 
of real-valued functions of bounded variation on $[0,T]$
such that 
\[ 
(\phi_{j},\phi_{k})_{a,b} 
 = \begin{cases} 
   0    &,\quad  j\not=k\\
   1    &,\quad  j=k 
\end{cases} .
\]
Then for each $v \in L_{a,b}^2[0,T] $, 
the Paley-Wiener-Zygmund(PWZ) stochastic integral 
$\langle{v,x}\rangle$ is defined by the formula
\[ 
\langle{v,x}\rangle
  = \lim_{n\to\infty} \int_{0}^{T}\sum_{j=1}^{n}
        (v,\phi_{j})_{a,b}\phi_{j}(t)dx(t)
\]
for all $x \in C_{a,b}[0,T]$ for which the limit exists; 
one  can show that for each $v \in L_{a,b}^2[0,T]$,
the PWZ stochastic integral $\langle{v,x}\rangle$ exists for 
$\mu$-a.e. $x \in C_{a,b}[0,T]$ and 
if $v$ is of bounded variation on $[0,T]$, 
then the PWZ  stochastic integral $\langle{v,x}\rangle$ 
equals the Riemann-Stieltjes integral $\int_{0}^{T}v(t)dx(t)$ 
for s-a.e. $x \in C_{a,b}[0,T]$.
For more details, see \cite{CS03}.

\begin{remark} 
For each $v \in L_{a,b}^2[0,T]$, the PWZ stochastic integral 
$\langle{v,x}\rangle$ is a Gaussian random variable on $C_{a,b}[0,T]$ 
with mean $\int _{0}^{T}v(s)da(s)$ and 
variance $\int_{0}^{T}v^{2}(s)db(s)$.
Note that for all $u,v \in L_{a,b}^2[0,T]$,
\[
\int_{C_{a,b}[0,T]} 
  \langle{u,x}\rangle\langle{v,x}\rangle   d \mu(x)
= \int _{0}^{T}u(s)v(s)db(s) 
 + \int_{0}^{T}u(s)da(s) \int_{0}^{T}v(s)da(s).
\]
Hence we see that for all $u,v \in L_{a,b}^2[0,T]$, 
$\int_{0}^{T}u(s)v(s)db(s)= 0$
if and only if $\langle{u,x\rangle}$ and $\langle{v,x\rangle}$ 
are independent random variables.
\end{remark}

\begin{remark}\label{remark:skoug01}
Recall that above, 
as well as in papers \cite{CS03, CCS03, CCS04, CCS07, 
CCS10, CCS09, CC12},
we require that $a:[0,T]\to \mathbb R$ 
is an absolutely continuous function 
with $a(0)=0$ and with $\int_0^T |a'(t)|^2 d t<\infty$.
Now throughout this paper we add the requirement
\begin{equation}\label{eq:condi-mean01}
\int_0^T|a'(t)|^2 d |a|(t)<\infty.
\end{equation}
\end{remark}

\begin{remark}\label{remark:skoug02}
Note that the function $a(t)=t^{2/3}$,
$0\le t\le T$ doesn't satisfy condition \eqref{eq:condi-mean01}
even though its derivative is an element of $L^2[0,T]$.
\end{remark}

\begin{remark}\label{remark:skoug03}
The function $a:[0,T]\to \mathbb R$  satisfies 
the requirements in Remark \ref{remark:skoug01} 
if and only if
the function $a'$ is an element of $L_{a,b}^2[0,T]$.
\end{remark}

\par
The following Cameron-Martin like subspace of $C_{a,b}[0,T]$
plays an important role throughout this paper.

\par
Let
\[
C_{a,b}'[0,T]
 =\bigg\{ w \in C_{a,b}[0,T] : w(t)=\int_0^t z(s) d b(s) 
       \hbox{  for some   } z \in L_{a,b}^2[0,T]  \bigg\}.        
\]
For $w\in C_{a,b}'[0,T]$, 
with $w(t)=\int_0^t z(s) d b(s)$ for $t\in [0,T]$, 
let $D_t : C_{a,b}'[0,T] \to L_{a,b}^2[0,T]$ 
be defined by the formula 
\begin{equation}\label{eq:Dt}
D_t w= z(t)=\frac{w'(t)}{b'(t)}. 
\end{equation}
Then $C_{a,b}' \equiv C_{a,b}'[0,T]$ with inner product
\[
(w_1, w_2)_{C_{a,b}'}
=\int_0^T  D_t w_1  D_t w_2  d b(t) 
=\int_0^T z_1(t)z_2(t) d b(t)    
\]
is a separable  Hilbert space.

\par
Note that  the linear operator given by equation \eqref{eq:Dt} 
is a  homeomorphism.
In fact, the inverse operator 
$D_t^{-1} : L_{a,b}^2[0,T] \to C_{a,b}'[0,T]$ is given by
\[
D_t^{-1} z=\int_0^t z(s) d b(s).			 
\]
It is easy to show that  $D_t^{-1}$ is a bounded operator
since
\[
\begin{split}
\|D_t^{-1}z\|_{C_{a,b}'}
&=\bigg\|\int_0^t z(s) d b(s)\bigg\|_{C_{a,b}'}
=\bigg(\int_0^T z^2(t)db(t)\bigg)^{\frac12}\\
&\le\bigg(\int_0^T z^2(t)d[b(t)+|a|(t)]\bigg)^{\frac12}=\|z\|_{a,b}.
\end{split}
\]
Applying the open mapping theorem, 
we see that $D_t$ is also bounded 
and there exist positive real numbers $\alpha$ and $\beta$ 
such that
$\alpha \| w\|_{C_{a,b}'} \le \| D_t w \|_{a,b}\le \beta \| w\|_{C_{a,b}'}$
for all $w\in C_{a,b}'[0,T]$.
Hence we see that  the Borel $\sigma$-algebra 
on $(C_{a,b}'[0,T], \|\cdot\|_{C_{a,b}'})$ is given by
\[
 \mathcal{B}(C_{a,b}'[0,T])
 =  \{D_t^{-1} (E) : E \in \mathcal{B}(L_{a,b}^2[0,T])\}.
\]

\begin{remark}\label{remark:skoug04}
Our conditions on $b:[0,T]\to \mathbb R$
imply that $0<\delta <b'(t)<   M$
for some positive real numbers $\delta$ and $M$, and  all $t\in[0,T]$.
\end{remark}

\par
The following lemma follows quite easily
from Remarks \ref{remark:skoug01}, 
\ref{remark:skoug03} and \ref{remark:skoug04} above
and the fact that 
$a(t)=\int_0^t \frac{a'(s)}{b'(s)}d b(s)$ on $[0,T]$.

\begin{lemma}
The function  $a:[0,T]\to \mathbb R$
satisfies the conditions in Remark \ref{remark:skoug01}
if and only if $a$ is an element of $C_{a,b}'[0,T]$.
\end{lemma}

\par
Throughout this paper for $w\in C_{a,b}'[0,T]$, 
with $w(t)=\int_0^t z(s) d b(s)$ for $t\in [0,T]$,
we will use the notation $(w,x)^{\sim}$ instead of 
$\langle{z,x\rangle}=\langle{D_t w,x\rangle}$. 
Then we have the following assertions.
\begin{itemize}
\item[(i)]
For each $w\in  C_{a,b}'[0,T]$, 
the random variable $x \mapsto (w,x)^{\sim}$
is Gaussian with mean $(w,a)_{C_{a,b}'}$ 
and variance $\|w\|_{C_{a,b}'}^2$.
\item[(ii)]
$(w,\alpha x)^{\sim}=( \alpha w, x)^{\sim}=\alpha (w,x)^{\sim}$ 
for any real number $\alpha$,   $w\in C_{a,b}'[0,T]$ 
and $x \in C_{a,b}[0,T]$.
\item[(iii)]
If $\{ w_1,   \ldots, w_n\}$ is an orthonormal set in $C_{a,b}'[0,T]$, 
then the random variables $(w_i,x)^{\sim}$'s are independent.
\end{itemize}

\par 
We denote the function space integral of a 
$\mathcal{W}(C_{a,b}[0,T])$-measurable functional $F$ by
\[ 
E[F]\equiv E_x[F(x)]= \int_{C_{a,b}[0,T]}F(x)d\mu (x)
\]
whenever the integral exists.

\begin{remark}\label{remark:skoug05}
For each $t\in[0,T]$, let
\[
\beta_t(s)=\int_0^s \chi_{[0,t]}(\tau) d b(\tau)
=\begin{cases}
    b(s), \quad & 0\le s \le t\\
    b(t), \quad & t\le s \le T
\end{cases}.
\]
Then the family of functions $\{\beta_t: 0\le t\le T\}$
from $C_{a,b}'[0,T]$ has the reproducing property
\[
(w,\beta_t)_{C_{a,b}'}=w(t)
\]
for all $w\in C_{a,b}'[0,T]$.
Note that $\beta_t(s)=\min\{b(s), b(t)\}$,
the covariance function associated with the 
generalized Brownian motion $Y$ used in this paper.
We also note that for each $x\in C_{a,b}[0,T]$,
\[
x(t)=\int_0^T \chi_{[0,t]}(s)dx(s)=(\beta_t,x)^{\sim}.
\]
\end{remark}

\begin{remark}
Let $A:C_{a,b}'[0,T]\to C_{a,b}'[0,T]$ be
a bounded linear operator with adjoint $A^*$.
Let $w$ and $g$ be elements of $C_{a,b}'[0,T]$ 
with $g(t)=\int_0^t z(s)db(s)$
for some $z \in L_{a,b}^2[0,T]$. Then
\[
\begin{split}
E_x[(A w,x)^{\sim}]
 &=E_x[\langle{ D_t A w,x}\rangle]
  =\int_0^T D_tA wd a (t)
  =\int_0^T D_tA w \frac{a' (t)}{b'(t)}d b(t)\\
 &=\int_0^T D_tA w D_ta d b(t)
  =(A w,a)_{C_{a,b}'}
  =( w,A^* a)_{C_{a,b}'}
\end{split}
\]
and
\[
\begin{split}
(A w,g)^{\sim}
 &=\langle{D_tA w, g}\rangle
  =\int_0^T D_tA w dg(t) 
  =\int_0^T D_tA w \frac{g'(t)}{b'(t)} d b(t)\\
 &=\int_0^T D_tA w D_tg (t) d b(t)
  =(A w,g)_{C_{a,b}'}
  =( w,A^* g)_{C_{a,b}'}.
\end{split}
\]

\par
Next, letting $A$ be the identity operator, 
yields the formulas
\[
E_x[(w,x)^{\sim}]
=\int_0^T D_tg d a(t)=(g,a)_{C_{a,b}'}
\hbox{\,\, and\,\, }
(w,g)^{\sim}=(w,g)_{C_{a,b}'}.
\]
\end{remark}

\par
In this paper,  as possible, we adopt 
the definitions and notations of  \cite{CS03, CCS04, CC12}
for the definitions of 
the generalized analytic Feynman integral  and the GFFT
of functionals on  $C_{a,b}[0,T]$.

\par
The following integration formula is used several times in this paper:
\begin{equation}\label{eq:int-formula}
\int_\mathbb{R} \exp \{-\alpha u^2+ \beta u \} du 
= \sqrt{\frac{\pi}{\alpha}} 
 \exp \Big\{ \frac{\beta^2}{4\alpha}  \Big\}
\end{equation}
for  complex numbers $\alpha $ and $\beta$ with $\hbox{\rm Re}(\alpha)> 0$.

\setcounter{equation}{0}
\section{The GFFT of functionals in a Banach algebra $\mathcal F_{A_1, A_2}^{\,\,a,b}$}\label{sec:3}

Let $\mathcal{M}(C_{a,b}'[0,T])$ be the space 
of  complex-valued, countably additive (and hence finite)  
Borel measures on $C_{a,b}'[0,T]$.
$\mathcal{M}(C_{a,b}'[0,T])$ is a Banach algebra 
under the total variation norm and with 
convolution as multiplication.

\par
We define the Fresnel type class $\mathcal F(C_{a,b}[0,T])$ 
of functionals on $C_{a,b}[0,T]$  as the space 
of all stochastic Fourier transforms of elements 
of $\mathcal{M}(C_{a,b}'[0,T])$; that is, 
$F \in \mathcal F (C_{a,b}[0,T])$ 
if and only if 
there exists a measure $f$ in $\mathcal M(C_{a,b}'[0,T])$
such that 
\begin{equation}\label{eq:element}
F(x) =\int_{C_{a,b}'[0,T]}\exp\{i(w,x)^{\sim}\} d f(w)                             
\end{equation}
for s-a.e. $x\in C_{a,b}[0,T]$.
More precisely, since we shall identify
functionals which coincide s-a.e. on $C_{a,b}[0,T]$,
$\mathcal{F}(C_{a,b} [0,T])$  can be   regarded 
as the space of all s-equivalence classes of 
functionals of the form \eqref{eq:element}.

\par
The Fresnel type class $\mathcal F(C_{a,b}[0,T])$ 
is a Banach algebra  with norm
\[
\|F\|=\|f\|=\int_{C_{a,b}'[0,T]}d|f|(w).
\]
In fact, the correspondence $f \mapsto F$
is injective, carries convolution into pointwise multiplication 
and is a Banach algebra isomorphism 
where $f$ and $F$ are related by \eqref{eq:element}.

\begin{remark}\label{re:physics}
{\rm (1)} 
The Banach algebra $\mathcal{F}(C_{a,b}[0,T])$ contains 
several  interesting functions which arise naturally 
in quantum mechanics:
Let $\mathcal{M}(\mathbb R)$ be the class 
of   $\mathbb C$-valued  countably additive measures on $\mathcal{B}(\mathbb R)$, 
the Borel class of $\mathbb R$.
For $\nu\in \mathcal{M}(\mathbb R)$, 
the Fourier transform $\widehat{\nu}$ of $\nu$
is a complex-valued function defined on $\mathbb R$ 
by the formula
\[
\widehat{\nu}(u)=\int_{\mathbb R}\exp\{i uv\} d \nu(v).
\]

\par
Let $\mathcal{G}$ be the set of all  complex-valued functions 
on  $[0,T]\times \mathbb R$ of the form
$\theta(s,u)=\widehat{\sigma}_s (u)$
where $\{ \sigma_s : 0\le s\le T\}$ 
is a family from $\mathcal{M}(\mathbb R)$ 
satisfying the following two conditions:
\begin{enumerate}
   \item[(i)]  
For every $E \in \mathcal{B}(\mathbb R)$, $\sigma_s (E)$ 
is Borel measurable in $s$, 
   \item[(ii)]  
$\int_0^T \| \sigma_s\| d b(s)< +\infty  $.  
\end{enumerate}

\par
Let $\theta \in \mathcal{G}$ and let $H$ be given by
\[
H(x)=\exp\bigg\{\int_0^T \theta(t,x(t)) d t\bigg\}
\]
for s-a.e. $x\in C_{a,b}[0,T]$.
It was shown  in \cite{CCL09} that 
the function $\theta(t,u)$ is Borel-measurable 
and  that $\theta(t,x(t))$,  $\int_0^T \theta(t,x(t)) d t$  
and $H(x)$ are elements of $\mathcal{F}(C_{a,b}[0,T])$.
This fact is relevant to quantum mechanics 
where exponential functions play a prominent role. 
For more details, see \cite{CCL09}.

{\rm (2)}
In \cite{CS03} Chang and Skoug introduced 
a Banach algebra $\mathcal S(L_{a,b}^2[0,T])$
of functionals on $C_{a,b}[0,T]$ given by
\[
\mathcal{S}(L_{a,b}^2[0,T])
=\bigg\{F: F(x)\approx\int_{L_{a,b}^2[0,T]}
            \exp\{i\langle{v,x\rangle}\}d \sigma(v),
            \, \sigma\in \mathcal M (L_{a,b}^2[0,T]) \bigg\},
\]
and then showed that the generalized analytic Feynman integral 
and the GFFT exist for $F\in \mathcal{S}(L_{a,b}^2[0,T])$ 
under appropriate conditions.
If  
\[
F(x) \approx
 \int_{L_{a,b}^2[0,T]}\exp\{i\langle{v,x\rangle}\}d \sigma(v)
\]
for some $\sigma\in \mathcal M (L_{a,b}^2[0,T])$, 
then we have 
\[
F(x) \approx
\int_{C_{a,b}'[0,T]}\exp\{i(w,x)^{\sim} \}d (\sigma\circ D_t)(w)
\]
where $D_t : C_{a,b}'[0,T]\to L_{a,b}^2[0,T]$ 
is given by equation \eqref{eq:Dt} above.
Conversely, if
\[
F(x)  \approx
\int_{C_{a,b}'[0,T]}\exp\{i(w,x)^{\sim}\}d f(w)
\]
for some $f\in \mathcal M (C_{a,b}'[0,T])$, 
then we have 
\[
F(x) \approx
 \int_{L_{a,b}^2[0,T]}\exp\{i\langle{v,x\rangle} \}d (f\circ D_t^{-1})(v).
\]
Thus we have that $F\in \mathcal{S}(L_{a,b}^2[0,T])$ 
if and only if $F\in \mathcal F (C_{a,b}[0,T])$.

\par
When $a(t)\equiv 0$ and $b(t)=t$ on $[0,T]$,
$\mathcal{S}(L_{a,b}^2[0,T])$ reduces to the 
Banach algebra $\mathcal{S}$ introduced by Cameron and Storvick \cite{CS80}.
In \cite[pp.609--629]{JL00}, Johnson and Lapidus
give a very complete summary of various relationships
which exist among  the Banach algebras $\mathcal{S}$, $\mathcal{F}(H)$
and $\mathcal{F}(B)$.
\end{remark}

\par
Let $A$ be a nonnegative self-adjoint operator on $C_{a,b}'[0,T]$ 
and $f$ any  complex measure on $C_{a,b}'[0,T]$. 
Then the functional 
\begin{equation}\label{eq:F-A}
F(x)
=\int_{C_{a,b}'[0,T]}\exp\{i(A^{1/2}w,x)^{\sim}\} d f(w)
\end{equation}
belongs to $\mathcal F (C_{a,b}[0,T])$  
because it can be rewritten as
$\int_{C_{a,b}'[0,T]}\exp\{i( w,x)^{\sim}\} d f_A(w)$ 
for $f_A=f \circ(A^{1/2})^{-1}$.
Let $A$ be self-adjoint  but not nonnegative. 
Then $A$ has the form
\begin{equation}\label{eq:decomposition}
A=A^+-A^- 
\end{equation}
where both $A^+$ and $A^-$ are bounded, nonnegative self-adjoint operators.

\par
In this section we will extend the ideas of \cite{KB84}
to obtain  expressions of the generalized analytic Feynman integral 
and the GFFT  of functionals of the form \eqref{eq:F-A} 
when $A$ is no longer required to be nonnegative.
To do this, we  will introduce   definitions and notations 
analogous to  those in \cite{CS03, CCS04, CC12}.

\par
Let $\mathcal{W}(C_{a,b}^2[0,T])$ denote the class of all Wiener 
measurable subsets of the product function space
$C_{a,b}[0,T]\times C_{a,b}[0,T]\equiv C_{a,b}^2[0,T]$.
A subset $B$ of   $C_{a,b}^{2}[0,T]$ 
is said to be  scale-invariant measurable
provided  $\{(\rho_1 x_1, \rho_2 x_2): (x_1, x_2)\in B\}$ 
is $\mathcal{W}(C_{a,b}^2[0,T])$-measurable
for every $\rho_1>0$ and $\rho_2>0$, 
and a scale-invariant measurable subset  $N$ of $C_{a,b}^{2}[0,T]$ 
is said  to be scale-invariant null  provided 
$(\mu\times\mu)(\{(\rho_1 x_1, \rho_2 x_2): (x_1, x_2)\in N\} )=0$
for every $\rho_1>0$ and $\rho_2>0$. 
A property that holds except on a scale-invariant null set 
is  said to hold s-a.e. on $C_{a,b}^2[0,T]$.
A functional $F$ on  $C_{a,b}^2[0,T]$ is said to be 
scale-invariant measurable  provided 
$F$ is defined on a scale-invariant measurable set 
and $F(\rho_1\,\,\cdot\,, \rho_2\,\,\cdot\,)$ 
is $\mathcal{W}(C_{a,b}^2[0,T])$-measurable 
for every $\rho_1>0$ and $\rho_2>0$.
If two functionals $F$ and $G$ defined on $C_{a,b}^2[0,T]$ 
are equal s-a.e., 
then we write $F\approx G$.
For more details, see \cite{CKY00, KB84}.

\par
We denote the product function space integral of a 
$\mathcal{W}(C_{a,b}^2[0,T])$-measurable functional $F$ by
\[ 
E[F]
  \equiv E_{\vec x}[F(x_1, x_2)]
  = \int_{C_{a,b}^2[0,T]}F(x_1,x_2)d(\mu\times\mu) (x_1,x_2)
\]
whenever the integral exists.

\par
Throughout this paper,
let $\mathbb C$, ${\mathbb C}_+$ and $\widetilde{\mathbb C}_+$ 
denote the complex numbers, the complex numbers with positive real part
and the nonzero complex numbers  with nonnegative real part, 
respectively.
Furthermore, for all $\lambda \in \widetilde{\mathbb C}_+$, 
$\lambda^{-1/2} (\hbox{or } \lambda^{1/2})$
is always chosen to have positive real part.
We also  assume that every functional $F$ on $C_{a,b}^2[0,T]$ 
we consider is s-a.e. defined and is scale-invariant measurable.

\begin{definition}\label{def:gfi}
Let $\mathbb{C}_+^2=\{ \vec\lambda =(\lambda_1,\lambda_2) \in \mathbb C^{2} : 
                  \text{\rm Re} (\lambda_j)  > 0 \hbox{ for } j=1,2 \}$ 
and let
$\widetilde{\mathbb C}_+^2=\{ \vec\lambda=(\lambda_1 , \lambda_2) \in \mathbb C^{2} : 
                  \lambda_j \ne 0 \hbox{ and }  \text{\rm Re} (\lambda_j)\ge 0    
                  \hbox{ for } j=1,2 \}$.
Let $F: C_{a,b}^2[0,T] \to \mathbb C $ be such that 
for each $\lambda_1>0$ and $ \lambda_2>0 $,
the function space integral
\[
J(\lambda_1, \lambda_2)=
 \int_{C_{a,b}^2[0,T]}F(\lambda_1^{-1/2}x_1, \lambda_2^{-1/2}x_2) 
   d (\mu\times\mu)(x_1, x_2)\\
\]
exists.
If there exists a function $J^*(\lambda_1,\lambda_2)$  
analytic in ${\mathbb C}_+^2$ such that  
$J^*(\lambda_1,\lambda_2)=J(\lambda_1,\lambda_2)$ 
for all $\lambda_1>0$ and $\lambda_2 >0$,
then $J^*(\lambda_1,\lambda_2)$ is defined to be 
the  analytic function space  integral of $F$ 
over $C_{a,b}^2[0,T]$  with parameter $\vec\lambda = (\lambda_1,\lambda_2)$, 
and for $\vec\lambda \in \mathbb C_+^2$ we write 
\[
E^{\text{\rm an}_{\vec\lambda}}[F] 
\equiv E_{\vec x}^{\text{\rm an}_{\vec\lambda}}[F(x_1, x_2)]
\equiv E_{x_1,x_2}^{\text{\rm an}_{(\lambda_1,\lambda_2)}}[F(x_1, x_2)]
= J^*(\lambda_1,\lambda_2).  
\]
Let $q_1$ and $q_2$ be nonzero real numbers. 
Let $F$ be a functional such that
$E^{\text{\rm an}_{\vec\lambda}}[F]$ exists 
for all $\vec\lambda \in \mathbb C_+^2$.
If the following limit exists, 
we call it the generalized analytic Feynman integral of $F$ 
with parameter $\vec q=(q_1,q_2)$ 
and we write
\[
E^{\text{\rm  anf}_{\vec q}}[F]
 \equiv E_{\vec x}^{\text{\rm  anf}_{\vec q}}[F(x_1, x_2)]
 \equiv E_{x_1,x_2}^{\text{\rm  anf}_{(q_1,q_2)}}[F(x_1, x_2)]
  =\lim_{\vec\lambda \to -i \vec q } E^{\text{\rm  an}_{\vec\lambda}}[F]
\]
where $\vec\lambda=(\lambda_1,\lambda_2) \to -i \vec q=(-iq_1,-iq_2)$ 
through values in $\mathbb C_+^2$.
\end{definition}

\begin{definition}\label{def:gfft}
Let $q_1$ and $q_2$ be nonzero real numbers.
For $\vec\lambda=(\lambda_1,\lambda_2) \in {\mathbb C}_+^2$ 
and $(y_1,y_2) \in C_{a,b}^2[0,T]$, 
let 
\[
T_{\vec\lambda}(F)(y_1,y_2)
\equiv T_{(\lambda_1,\lambda_2)}(F)(y_1,y_2)
=E_{\vec x}^{\text{\rm  an}_{\vec\lambda}}[F(y_1+x_1 , y_2+x_2)].          
\]
For $p\in (1,2]$, we define the $L_p$ analytic GFFT, 
$T_{\vec q}^{ (p)}(F)$ of $F$, by the formula  $(\vec \lambda \in {\mathbb C}_+^2)$
\[
 T_{\vec q}^{ (p)}(F) (y_1,y_2)
 \equiv T_{(q_1,q_2)}^{ (p)}(F) (y_1,y_2)
 =\operatorname*{l.i.m.}\limits_{\vec\lambda\to -i\vec q }
 T_{\vec\lambda} (F)(y_1,y_2)    	             
\]
if it exists; i.e.,  for each $\rho_1 >0$ and $\rho_2 >0$,
\[
\lim_{\vec\lambda\to -i\vec q }
 \int_{C_{a,b}^2[0,T]} 
 \big| T_{\vec\lambda} (F)(\rho_1  y_1 , \rho_2 y_2)
  -T_{\vec q}^{(p)}(F)(\rho_1  y_1 , \rho_2 y_2) \big|^{p'} 
  d(\mu\times \mu)(y_1,y_2)=0
\]
where $1/p +1/p' =1$. 
We define the $L_1$ analytic GFFT, 
$T_{\vec q}^{ (1)}(F)$ of $F$, by the formula $(\vec\lambda \in \mathbb C_+^2)$
\[
 T_{\vec q}^{ (1)}(F) (y_1,y_2)
 = \lim_{\vec\lambda \to  -i \vec q} T_{\vec\lambda} (F)(y_1 , y_2)  				 
\]
if it exists.
\end{definition}

\par
We note that for $1\le p \le 2$, $T_{\vec q}^{ (p)}(F)$ is defined only  s-a.e..
We also note that if $T_{\vec q}^{ (p)}(F)$ exists and if $F\approx G$, 
then $T_{\vec q}^{ (p)}(G)$ exists and 
$ T_{\vec q}^{ (p)}(G) \approx  T_{\vec q}^{ (p)}(F)$.
Moreover, from Definition \ref{def:gfft}, 
we see that for $q_1, q_2\in \mathbb R-\{0\}$,
\begin{equation}\label{eq:t1q0}
E_{\vec x}^{\text{\rm anf}_{\vec q}}[F(x_1,x_2)]=T_{\vec q}^{(1)}(F)(0,0).
\end{equation}

\par
Next we give the definition of the generalized Fresnel type class 
$\mathcal{F}_{A_1,A_2}^{\,\,a,b}$.

\begin{definition}\label{def:class}
Let $A_1$ and $A_2$ be bounded, nonnegative self-adjoint operators 
on $C_{a,b}'[0,T]$.
The generalized Fresnel type class  $\mathcal{F}_{A_1,A_2}^{\,\,a,b}$ 
of functionals on $C_{a,b}^2[0,T]$ 
is defined as the space of all functionals $F$ 
on $C_{a,b}^2[0,T]$ of the form
\begin{equation}\label{eq:element2}
F(x_1,x_2)=\int_{C_{a,b}'[0,T]} 
  \exp\bigg\{\sum\limits_{j=1}^2i(A_j^{1/2}w,x_j)^{\sim}\bigg\} d f (w)      
\end{equation}
for some  $f\in \mathcal M(C_{a,b}'[0,T])$. 
More precisely, since we identify functionals 
which coincide s-a.e. on $C_{a,b}^2[0,T]$,
$\mathcal{F}_{A_1,A_2}^{\,\,a,b}$ can be regarded as 
the  space of all s-equivalence classes 
of functionals of the form \eqref{eq:element2}.
\end{definition}

\begin{remark}\label{re:A}
{\rm (1)} 
In Definition \ref{def:class} above, 
let $A_1$  be the identity operator on $C_{a,b}'[0,T]$
and $A_2 \equiv 0$. 
Then $\mathcal{F}_{A_1,A_2}^{\,\,a,b}$ is 
essentially the Fresnel type class  $\mathcal F(C_{a,b}[0,T])$ 
and for $p\in[1,2]$ and nonzero real numbers $q_1$ and $q_2$,
\[
T_{(q_1,q_2)}^{(p)}(F)(y_1,y_2)=T_{q_1}^{(p)} (F_0)(y_1),
\]
if it exists, where $F_0(x_1)=F(x_1,x_2)$ 
for all $(x_1,x_2)\in C_{a,b}^2[0,T]$ 
and $T_{q_1}^{(p)}(F_0)(y)$ means the $L_p$  analytic GFFT 
on $C_{a,b}[0,T]$, see \cite{CS03, CCS03}.

{\rm (2)}  
The map $f \mapsto F$ defined by \eqref{eq:element2} 
sets up an algebra isomorphism 
between $\mathcal M(C_{a,b}'[0,T])$ and  $\mathcal F_{A_1,A_2}^{\,\,a,b}$ 
if $\text{\rm Ran}(A_1+A_2)$ is dense in $C_{a,b}'[0,T]$  
where $\text{\rm Ran}$  indicates the range of an operator.
In this case $\mathcal{F}_{A_1,A_2}^{\,\,a,b}$ becomes a Banach algebra 
under the norm $\|F\|=\|f\|$. 
For more details,   see  \cite{KB84}. 
\end{remark}

\begin{remark}\label{re:effect-a}
Let $F$ be given by equation \eqref{eq:element2}.
In evaluating $E_{\vec x}[F(\lambda_1^{-1/2}x_1,\lambda_2^{-1/2} x_2)]$ 
and  $T_{(\lambda_1,\lambda_2)}(F)(y_1,y_2)
  =E_{\vec x}[F(y_1+\lambda_1^{-1/2}x_1,y_2+\lambda_2^{-1/2} x_2)]$ 
for $\lambda_1>0$ and $\lambda_2>0$,
the expression
\begin{equation}\label{eq:h-function}
\begin{split}
\psi(\vec\lambda;\vec A;w)
&\equiv \psi(\lambda_1,\lambda_2;A_1,A_2;w)\\
&=\exp\bigg\{\sum_{j=1}^2\bigg[-\frac{(A_jw,w)_{C_{a,b}'}}{2\lambda_j} 
+i\lambda_j^{-1/2}(A_j^{1/2}w,a)_{C_{a,b}'}\bigg]\bigg\}
\end{split}
\end{equation}
occurs.
Clearly, for $\lambda_j>0$, $j\in\{1,2\}$, 
$|\psi(\vec\lambda;\vec A;w)|\le 1$  for all $w\in C_{a,b}'[0,T]$.
But for $\vec\lambda \in \widetilde{\mathbb  C}_+^2$, 
$|\psi(\vec\lambda;\vec A;w)|$ is not necessarily  bounded by $1$.

\par
Note that for each $\lambda_j \in \widetilde{\mathbb  C}_+$ 
with  $\lambda_j =\alpha_j+i\beta_j$, $j\in\{1,2\}$,
\[
\lambda_j^{1/2}
= \sqrt{\tfrac{\sqrt{\alpha_j^2 + \beta_j^2} +\alpha_j}{2}} 
  + i \frac{\beta_j}{|\beta_j|}
     \sqrt{\tfrac{\sqrt{\alpha_j^2+ \beta_j^2} 
             -\alpha_j}{2}},
\]
\[
\text{\rm  Re}(\lambda_j^{1/2})
 \ge |\text{\rm  Im }(\lambda_j^{1/2})|
 \ge 0,
\]
and
\[
\lambda_j^{-1/2}
= \sqrt{\tfrac{\sqrt{\alpha_j^2 + \beta_j^2}
                        +\alpha_j}{2(\alpha_j^2 + \beta_j^2)}} 
   -i \frac{\beta_j}{|\beta_j|}
     \sqrt{\tfrac{\sqrt{\alpha_j^2+ \beta_j^2} 
          -\alpha_j}{2(\alpha_j^2 + \beta_j^2)}}.
\]
Hence, for $\lambda_j \in  \widetilde{\mathbb  C}_+$ with  
$\lambda_j =\alpha_j+i\beta_j$, $j\in\{1,2\}$,
\begin{equation}\label{eq:vert-psi}
\begin{split}
&|\psi(\vec \lambda;\vec A; w)|\\
&=\bigg|\exp\bigg\{ 
   \sum_{j=1}^2\bigg[-\frac{1}{2\lambda_j}(A_jw,w)_{C_{a,b}'}  
   +i\lambda_j^{-1/2} (A_j^{1/2}w,a)_{C_{a,b}'}\bigg] \bigg\}\bigg|\\
&=\bigg|\exp\bigg\{ 
   \sum_{j=1}^2\bigg[  -\frac{1}{2 } \bigg(\frac{\alpha_j}{\alpha_j^2 + \beta_j^2}
    -i \frac{\beta_j}{\alpha_j^2 + \beta_j^2}\bigg)(A_jw,w)_{C_{a,b}'} \\
& \qquad\qquad
   + i\bigg(\sqrt{\tfrac{\sqrt{\alpha_j^2 + \beta_j^2} 
        +\alpha_j}{2(\alpha_j^2 + \beta_j^2)}} 
   -i \frac{\beta_j}{|\beta_j|}\sqrt{\tfrac{\sqrt{\alpha_j^2 + \beta_j^2} 
    -\alpha_j}{2(\alpha_j^2 + \beta_j^2)}}\bigg)
  (A_j^{1/2}w,a)_{C_{a,b}'}\bigg]\bigg\}\bigg|\\
\end{split}
\end{equation}
\begin{equation*}
\begin{split}
&= \exp\bigg\{  
   \sum_{j=1}^2\bigg[   - \frac{\alpha_j}{2(\alpha_j^2 
     + \beta_j^2)}(A_jw,w)_{C_{a,b}'}\\
&\qquad\qquad\qquad\qquad
   + \frac{\beta_j}{|\beta_j|}
    \sqrt{\tfrac{\sqrt{\alpha_j^2 + \beta_j^2} 
    -\alpha_j}{2(\alpha_j^2 + \beta_j^2)}} 
  (A_j^{1/2}w,a)_{C_{a,b}'}\bigg]\bigg\}.
\end{split}
\end{equation*}

\par
The last expression of \eqref{eq:vert-psi} 
 is an unbounded function
of $w$ for $w\in C_{a,b}'[0,T]$.
Thus $E^{{\rm an}_{\vec  \lambda}}[F]$, $E^{{\rm anf}_{\vec q}}[F]$,
$T_{\vec \lambda}(F)$ and $T_{\vec q}^{(p)}(F)$
might not exist.
Throughout this paper
we thus will need to put additional restrictions
on the complex measure $f$ corresponding to $F$
in order to obtain our results for the GFFT 
and the generalized analytic Feynman integral of $F$.
\end{remark}

\par
In view of Remark \ref{re:effect-a} above, 
we clearly need to impose additional restrictions 
on the functionals $F$ in $\mathcal{F}^{\,\,a,b}_{A_1,A_2}$.

\par
For a positive real number $q_0$, 
let  
\begin{equation}\label{eq:domain}
\begin{split}
\Gamma_{q_0}
&=\bigg\{ \vec\lambda=(\lambda_1,\lambda_2) \in \widetilde{\mathbb C}_+^2  
    \,\bigg|\, \lambda_j=\alpha_j+i \beta_j ,\\
&\qquad\qquad\qquad
|\text{\rm  Im}(\lambda_j^{-1/2})| 
   = \sqrt{\tfrac{\sqrt{\alpha_j^2+\beta_j^2}
     -\alpha_j}{2 (\alpha_j^2+\beta_j^2)}} 
< \frac{1}{\sqrt{2q_0}},\,j=1,2 \bigg\}
\end{split}
\end{equation}
and  let
\begin{equation}\label{eq:Domi-1}
\begin{split}
k(q_0;\vec A;w)
&\equiv  k(q_0;A_1,A_2;w)\\
&=\exp\bigg\{  \sum_{j=1}^2(2q_0)^{-1/2} 
 \|A_j^{1/2}\|_{o}\|w\|_{C_{a,b}'}\|a\|_{C_{a,b}'}   \bigg\}
\end{split}
\end{equation}
where $\|A_j^{1/2}\|_o$ means the operator norm of $A_j^{1/2}$ 
for $j\in\{1,2\}$.
Then for all 
$\vec\lambda=(\lambda_1,\lambda_2) \in \Gamma_{q_0}$,
\begin{equation}\label{eq:inequality100}
\begin{split}
|\psi(\vec\lambda;\vec A;w)|
&= 
  \exp\bigg\{  \sum_{j=1}^2\bigg[   
    - \frac{\alpha_j}{2(\alpha_j^2 + \beta_j^2)}\|A_j^{1/2}w\|_{C_{a,b}'}^2\\
& \qquad \qquad\qquad
+ \frac{\beta_j}{|\beta_j|}\sqrt{\tfrac{\sqrt{\alpha_j^2 
       + \beta_j^2} -\alpha_j}{2(\alpha_j^2 + \beta_j^2)}} 
(A_j^{1/2}w,a)_{C_{a,b}'}\bigg]\bigg\}\\
& \le 
 \exp\bigg\{  \sum_{j=1}^2 \sqrt{\tfrac{\sqrt{\alpha_j^2 
    + \beta_j^2} -\alpha_j}{2(\alpha_j^2 + \beta_j^2)}}
|(A_j^{1/2}w,a)_{C_{a,b}'}|    \bigg\}\\
& \le
\exp\bigg\{  \sum_{j=1}^2 \sqrt{\tfrac{\sqrt{\alpha_j^2 
    + \beta_j^2} -\alpha_j}{2(\alpha_j^2 + \beta_j^2)}}
\| A_j^{1/2}w\|_{C_{a,b}'} \|a\|_{C_{a,b}'}  \bigg\}\\
&< k(q_0;\vec A;w).
\end{split}
\end{equation}

\par
We note that for all real $q_j$ with $|q_j|>q_0$, $j\in\{1,2\}$,
\[
(-iq_j)^{-1/2}=\frac{1}{\sqrt{|2q_j|}}
   + \hbox{\rm sign}(q_j)\frac{i}{\sqrt{|2q_j|}}
\]
and $(-iq_1,-iq_2)\in \Gamma_{q_0}$.

\par 
For the existence of the GFFT of $F$,
we  define a subclass $\mathcal{F}_{A_1,A_2}^{\, \,q_0}$
of $\mathcal{F}_{A_1,A_2}^{\,\,a,b}$ by
$F\in \mathcal{F}_{A_1,A_2}^{\,\, q_0}$  if and only if 
\begin{equation}\label{eq:finite000}
\int_{C_{a,b}'[0,T]}k(q_0;\vec A;w)d|f|(w)<+\infty   
\end{equation}
where $f$ and $F$ are  related by equation \eqref{eq:element2}
and $k$ is given by equation \eqref{eq:Domi-1}.

\begin{remark}
Note that in case $a(t)\equiv 0$ and $b(t)=t$ on $[0,T]$,
the function space $C_{a,b}[0,T]$ reduces   to 
the classical Wiener space $C_0[0,T]$
and $(w,a)_{C_{a,b}'}=0$ for all $w \in C_{a,b}'[0,T]=C_0'[0,T]$. 
Hence for all $\vec\lambda\in \widetilde{\mathbb C}_+^2$, 
$|\psi(\vec\lambda;\vec A;w)| \le 1$
and for any positive real number $q_0$, 
$\mathcal{F}_{A_1,A_2}^{\,\,q_0}=\mathcal{F}_{A_1,A_2}$,
the Kallianpur and Bromley's class
introduced in Section \ref{sec:1}.
\end{remark}

\begin{theorem}\label{thm:t1q}
Let $q_0$ be a positive real number  and 
let $F$ be an element of $\mathcal F_{A_1, A_2}^{\,\,q_0}$.
Then for any nonzero real numbers $q_1$ and $q_2$ 
with $|q_j|> q_0 $, $j\in\{1,2\}$, 
the  $L_1$ analytic GFFT of $F$, $T_{\vec q}^{(1)}(F)$ 
exists and  is given by  the formula
\begin{equation}\label{eq:tpq}
\begin{split}
&T_{\vec q}^{(1)} (F)(y_1,y_2) \\
&=\int_{C_{a,b}'[0,T]}
\exp\bigg\{ \sum\limits_{j=1}^2i(A_j^{1/2} w,y_j)^{\sim} \bigg\}
 \psi(-i\vec q;\vec A;w) d f( w)\\
\end{split}
\end{equation}
for s-a.e. $(y_1,y_2) \in C_{a,b}^2[0,T]$, 
where $\psi$ is  given by equation \eqref{eq:h-function}.
\end{theorem}
\begin{proof}
We first note that for $j\in\{1,2\}$,
the PWZ stochastic integral $(A_j^{1/2}w,x)^{\sim}$ is a Gaussian 
random variable with mean $(A_j^{1/2}w,a)_{C_{a,b}'}$ and 
variance $ \|A_j^{1/2}w\|_{C_{a,b}'}^2=(A_jw,w)_{C_{a,b}'}$.
Hence, using equation \eqref{eq:element2},  the Fubini theorem,  
the change of variables theorem  and equation \eqref{eq:int-formula}, 
we have that for all $\lambda_1>0$ and $\lambda_2>0$,
\[
\begin{split}
&J(y_1,y_2;\lambda_1, \lambda_2)\\
&\equiv E_{\vec x}[F(y_1+\lambda_1^{-1/2}x_1,
      y_2+\lambda_2^{-1/2}x_2)]\\
&= \int_{C_{a,b}'[0,T]}
   \exp\bigg\{\sum_{j=1}^2i(A_{j}^{1/2}w,y_j)^{\sim}\bigg\}  
   \bigg(\prod_{j=1}^2 E_{x_j}\big[ \exp  \big\{i\lambda_j^{-1/2} 
   (A_j^{1/2}w,x_j)^{\sim}   \big\} \big]\bigg)df(w) \\
&= \int_{C_{a,b}'[0,T]}  
    \exp\bigg\{\sum_{j=1}^2i(A_{j}^{1/2}w,y_j)^{\sim}\bigg\} 
    \bigg[\prod_{j=1}^2 \Big(  2\pi (A_jw,w)_{C_{a,b}'} \Big)^{-1/2}\\
&\qquad\qquad\qquad\qquad\times 
\int_{\mathbb R} 
   \exp \bigg\{i\lambda_j^{-1/2}u_j  
   -\frac{ [u_j-(A_j^{1/2}w,a)_{C_{a,b}'} ]^2 }{2 (A_jw,w)_{C_{a,b}'} }
      \bigg\} du_j\bigg] df(w)  \\
&=\int_{C_{a,b}'[0,T]} 
   \exp\bigg\{\sum_{j=1}^2i(A_{j}^{1/2}w,y_j)^{\sim}\bigg\} \\
&\quad\times
   \bigg[\prod_{j=1}^2\exp  \bigg\{ \frac{(A_jw,w)_{C_{a,b}'}}{2}   
   \bigg[i\lambda_j^{-1/2}    
    +\frac{(A_j^{1/2}w,a)_{C_{a,b}'}}{(A_jw,w)_{C_{a,b}'}}  \bigg]^2     
    -\frac{(A_j^{1/2}w,a)_{C_{a,b}'}^2}{2(A_j w,w)_{C_{a,b}'}}      
    \bigg\}\bigg]df(w)  \\
\end{split}
\]
\[
\begin{split}
&=\int_{C_{a,b}'[0,T]}
   \exp\bigg\{\sum_{j=1}^2i(A_{j}^{1/2}w,y_j)^{\sim }\bigg\}
\psi(\vec \lambda;\vec A;w)df(w).
\end{split}
\]
Let
\begin{equation}\label{eq:lambdastar}
T_{\vec\lambda}(F)(y_1,y_2) 
  =  \int_{C_{a,b}'[0,T]}
   \exp\bigg\{\sum_{j=1}^2i(A_{j}^{1/2}w,y_j)^{\sim }\bigg\}
   \psi(\vec \lambda;\vec A;w)df(w)  
\end{equation}
for each $\vec\lambda \in  \mathbb C_+^2$.
Clearly,  
\[
T_{\vec\lambda}(F)(y_1,y_2)
=J(y_1,y_2;\lambda_1, \lambda_2)
\] 
for all $\lambda_1 > 0$ and $\lambda_2 > 0$.
Let $\Gamma_{q_0}$ be given by equation \eqref{eq:domain}.
Then for all $\vec\lambda \in \hbox{\rm Int}(\Gamma_{q_0})$, 
\[
|T_{\vec\lambda}(F)(y_1,y_2)|
<\int_{C_{a,b}'[0,T]}k(q_0;\vec A;w)d|f|(w) < +\infty.
\]
Using this fact and the dominated convergence theorem, 
we see that $T_{\vec\lambda}(F)(y_1,y_2)$ is a continuous function 
of $\vec\lambda=(\lambda_1,\lambda_2)$ on $\hbox{\rm Int}(\Gamma_{q_0})$.
For each $w\in C_{a,b}'[0,T]$, 
$\psi(\vec\lambda;\vec A;w)$
is an analytic function of $\vec\lambda$ 
throughout the domain $\text{\rm Int} (\Gamma_{q_0})$
so that $\int_{\Delta}\psi(\vec\lambda;\vec A;w)d\vec \lambda=0$ 
for every rectifiable simple closed curve $\Delta$ 
in $\text{\rm Int}(\Gamma_{q_0})$.
By equation \eqref{eq:lambdastar}, 
the Fubini theorem and the Morera theorem, 
we see that  $T_{\vec\lambda}(F)(y_1,y_2)$ 
is an analytic function of $\vec\lambda$
throughout the domain $\text{\rm Int}(\Gamma_{q_0})$.
Finally, by the dominated convergence theorem,
it follows that 
\[
\begin{split}
T_{\vec q}^{(1)}(F)(y_1,y_2)
&=\lim_{\vec\lambda \to -i\vec q}T_{\vec \lambda}(F)(y_1,y_2) \\
& =  \int_{C_{a,b}'[0,T]}\lim_{\vec \lambda \to -i\vec q}  
\exp \bigg\{ \sum_{j=1}^2i(A_j^{1/2}w,y_j)^{\sim} \bigg\}
   \psi(\vec\lambda;\vec A;w)  df(w)\\
& = \int_{C_{a,b}'[0,T]} \exp \bigg\{  
   \sum_{j=1}^2i(A_j^{1/2}w,y_j)^{\sim}\bigg\}
\psi(-i\vec q;\vec A;w) df(w).
\end{split}
\]
\end{proof}

\begin{theorem}\label{thm:tpq}
Let $q_0$ and $F$  be as in Theorem \ref{thm:t1q}.
Then for all $p\in(1,2]$ and all nonzero real numbers 
$q_1$ and $q_2$ with $|q_j|> q_0$, $j\in\{1,2\}$,
the  $L_p$ analytic GFFT of $F$, 
$T_{\vec q}^{(p)}(F)$ exists   and  is given 
by the right hand side of  equation \eqref{eq:tpq}
for s-a.e. $(y_1,y_2) \in C_{a,b}^2[0,T]$. 
\end{theorem}
\begin{proof}
Let $\Gamma_{q_0}$ be given by equation \eqref{eq:domain}.
It was shown in the proof of Theorem \ref{thm:t1q}
that $T_{\vec\lambda}(F)(y_1,y_2)$ is an analytic function 
of $\vec\lambda$ throughout the domain $\text{\rm Int}(\Gamma_{q_0})$.
In view of Definition \ref{def:gfft}, 
it will  suffice  to show that for each $\rho_1>0$ and $\rho_2>0$,
\[
\lim\limits_{\vec\lambda \to -i\vec q}
\int_{C_{a,b}[0,T]}\big|T_{\vec\lambda}(F)(\rho_1 y_1,\rho_2y_2)
   -T_{\vec q}^{(p)}(F)(\rho_1 y_1,\rho_2y_2) \big|^{p'} 
   d (\mu\times\mu)(y_1,y_2)
=0.
\] 

\par
Fixing $p\in (1,2]$ and using the inequalities  
\eqref{eq:inequality100}  and \eqref{eq:finite000}, 
we obtain that   for all $\rho_j>0$, $j\in\{1,2\}$ 
and  all $\vec\lambda \in \Gamma_{q_0}$,
\[
\begin{split}
&\big|T_{\vec\lambda}(F)(\rho_1 y_1,\rho_2 y_2)
  -T_{\vec q}^{(p)}(F)(\rho_1 y_1,\rho_2 y_2) \big|^{p'}\\
&\le \bigg|\int_{C_{a,b}'[0,T]}  \exp
   \bigg\{ \sum_{j=1}^2i\rho_j(A_{j}^{1/2}w, y_j)^{\sim}\bigg\}
 \Big[ \psi(\vec\lambda;\vec A;w) -\psi(-i\vec q;\vec A;w) \Big]
d f(w)\bigg|^{p'}\\
&\le \Bigg(\int_{C_{a,b}'[0,T]}\Big[\big|\psi(\vec\lambda;\vec A;w)\big|  
   +\big|\psi(-i\vec q;\vec A;w)\big| \Big]d |f|(w)\Bigg)^{p'}\\
&\le\Bigg(2\int_{C_{a,b}'[0,T]} 
   k(q_0;\vec A;w)  d  |f| (w)\Bigg)^{p'}<+\infty .
\end{split}
\]
Hence by the dominated convergence theorem,  
we see that for each $p\in (1,2]$ and each  $\rho_1 >0$ and $\rho_2>0$,  
\[
\begin{split}
&\lim\limits_{\vec\lambda \to -i\vec q}
   \int_{C_{a,b}[0,T]}\big|T_{\vec\lambda}(F)(\rho_1 y_1,\rho_2 y_2)-
     T_{\vec q}^{(p)}(F)(\rho_1 y_1,\rho_2 y_2) \big|^{p'} 
     d (\mu\times\mu)(y_1,y_2)\\
&=\lim\limits_{\vec\lambda \to -i\vec q}
  \int_{C_{a,b}^2[0,T]}\bigg| 
  \int_{C_{a,b}'[0,T]} \exp\bigg\{\sum_{j=1}^2i(A_j^{1/2}w, \rho_j y_j)^{\sim} \bigg\}
  \psi(\vec\lambda;\vec A;w)  d f(w)\\
&\,\,\,   
 -\int_{C_{a,b}'[0,T]} \exp\bigg\{\sum_{j=1}^2i(A_j^{1/2}w, \rho_j y_j)^{\sim} \bigg\}
    \psi(-i\vec q;\vec A; w)  d f(w) \bigg|^{p'} d (\mu\times\mu)(y_1,y_2)\\
&= \int_{C_{a,b}[0,T]}\bigg| 
 \int_{C_{a,b}'[0,T]} \exp\bigg\{\sum_{j=1}^2i(A_j^{1/2}w, \rho_j y_j)^{\sim} \bigg\}\\
&\qquad\qquad\times
 \lim\limits_{\vec\lambda \to -i\vec q}
  \Big[\psi(\vec\lambda;\vec A;w) - \psi(-i\vec q;\vec A;w) \Big]  d f(w) \bigg|^{p'} 
  d (\mu\times\mu)(y_1,y_2)\\
&=0
 \end{split}
\]
which concludes the proof of Theorem \ref{thm:tpq}.
\end{proof}

\begin{remark}\label{re:remark}
{\rm (1)} 
In view of Theorems \ref{thm:t1q} and \ref{thm:tpq}, 
we see that for each $p\in[1,2]$, 
the $L_p$   analytic GFFT  of $F$, $T_{\vec q}^{(p)}(F)$ 
is given by the right hand side of  equation  \eqref{eq:tpq}
for $q_0$, $q_1$, $q_2$ and $F$ as in Theorem \ref{thm:t1q}, 
and   for s-a.e.   $(y_1,y_2)\in C_{a,b}^2[0,T]$,
\[
T_{\vec q}^{(p)} (F)(y_1,y_2)\\
=E_{\vec x}^{\text{\rm anf}_{\vec q}}
  [F(y_1+x_1, y_2+x_2)], \qquad  p\in[1,2].
\]
In particular, using this fact and  equation  \eqref{eq:t1q0}, 
we have that for all $p\in [1,2]$
\[
T_{\vec q}^{(p)}(F)(0,0)
=E_{\vec x}^{\text{\rm  anf}_{\vec q}}[F(x_1,x_2)].
\]

{\rm (2)} 
For nonzero real numbers $q_1$ and $q_2$ 
with $|q_j|>q_0$, $j\in\{1,2\}$,  define a set function
$f_{\vec q}^{\vec A}: \mathcal{B}(C_{a,b}'[0,T])\to \mathbb C$ 
by
\[
f_{\vec q}^{\vec A}(B)
=\int_{B}\psi(-i\vec q;\vec A;w)d f(w),
   \qquad B\in \mathcal{B}(C_{a,b}'[0,T]),
\]
where $f$ and $F$ are related by equation \eqref{eq:element2} 
and  $\mathcal{B}(C_{a,b}'[0,T])$ is the Borel $\sigma$-algebra 
of $C_{a,b}'[0,T]$.
Then it is obvious that $f_{\vec q}^{\vec A} $ 
belongs to $\mathcal{M}(C_{a,b}'[0,T])$ and for all $p\in [1,2]$, 
$T_{\vec q}^{(p)}(F)$  can be expressed as 
\[
T_{\vec q}^{(p)}(F)(y_1,y_2)
=\int_{C_{a,b}'[0,T]}
\exp\bigg\{\sum_{j=1}^2 i(A_j^{1/2}w,y_j)^{\sim} \bigg\}
  df_{\vec q}^{\vec A} (w)
\]
for s-a.e. $(y_1,y_2)\in C_{a,b}^2[0,T]$.
Hence $T_{\vec q}^{(p)}(F)$ belongs to 
$\mathcal{F}_{A_1.A_2}^{\,\,a,b}$ for all $p\in [1,2]$.

{\rm (3)} 
Let $A$ be a bounded self-adjoint operator on $C_{a,b}'[0,T]$.
Then $A$ has the form \eqref{eq:decomposition}.
Take $A_1=A_+$ and $A_2=A_-$ in Definition \ref{def:class} above. 
Then for $F\in \mathcal F_{A_+,A_-}^{\,\,q_0}$ 
and for real $q$ with $|q|> q_0$,
equations \eqref{eq:t1q0} and   \eqref{eq:tpq} with $\vec q=(q_1, q_2)$ 
replaced with $\vec q=(q,-q)$ becomes
\[
\begin{split}
&E_{\vec x}^{\text{\rm  anf}_{(q,-q)}}[F(x_1,x_2)]
=T_{(q,-q)}^{(p)} (F)(0,0)  \\
& =\int_{C_{a,b}'[0,T]}
\exp\bigg\{  
   -\frac{i}{2q} (Aw,w )_{C_{a,b}'}   
     \bigg\} d f_{\vec q}^{\vec A}( w) . 
\end{split}
\]
\end{remark}

\par
The following corollary follows from  
equations \eqref{eq:t1q0} and \eqref{eq:tpq}.

\begin{corollary}\label{coro:t1q-feynman}
Let $q_{0}$ and $F$ be as in Theorem \ref{thm:t1q}.
Then for all real numbers $q_1$ and $q_2$ 
with $|q_j|> q_{0}$, $j\in\{1,2\}$,
the generalized analytic Feynman integral 
$E^{\text{\rm anf}_{\vec q}} [F]$ 
of $F$ exists and is given by the formula
\[
E^{\text{\rm anf}_{\vec q}} [F] 
= \int_{C_{a,b}'[0,T]}\psi(-i\vec q;\vec A;w)df (w)
\]
where $\psi$ is given by equation \eqref{eq:h-function}.
\end{corollary}

\par
In the proof of Theorem \ref{thm:t1q}, 
we showed that $T_{\vec\lambda}(F)$ is 
an analytic function of $\vec\lambda$
throughout the domain  $\hbox{\rm Int}(\Gamma_{q_0})$.
Thus we have the following corollary.

\begin{corollary}
Let $q_{0}$ and $F$ be as in Theorem \ref{thm:t1q}
and let $\Gamma_{q_0}$ be given by \eqref{eq:domain}.
Then for each $\vec\lambda \in \hbox{\rm Int}(\Gamma_{q_0})$,
\[
E^{\text{\rm an}_{\vec\lambda}}[F] 
 = \int_{C_{a,b}'[0,T]} \psi(\vec\lambda;\vec A ;w)df(w)
\]
where $\psi$ is given by equation \eqref{eq:h-function}.
\end{corollary}

\setcounter{equation}{0}
\section{Relationships between the GFFT
and the function space integral of functionals 
in $\mathcal{F}_{A_1,A_2}^{\,\,a,b}$}\label{sec:4}

\par
In this section we establish a relationship  
between the GFFT  and the function space integral
of  functionals in the  Fresnel type class $\mathcal F_{A_1,A_2}^{\,\,a,b}$.

\par
Throughout this section, for convenience, 
we use the following notation:  
for given $\lambda \in \widetilde{\mathbb C}_+$ and $n =1,2,\ldots$, 
let
\begin{equation}\label{eq:gn}
G_n(\lambda,x)
= \exp \bigg\{    
\bigg[\frac{1-\lambda }{2 } \bigg]  
 \sum_{k=1}^{n}[(e_k,x)^{\sim}]^2
+ (\lambda^{1/2 }-1) \sum_{k=1}^{n} (e_k,a)_{C_{a,b}'}(e_k,x)^{\sim} \bigg\}  
\end{equation}
where $\{e_n\}_{n=1}^{\infty}$ is a complete orthonormal set 
in $C_{a,b}'[0,T]$.

\par
To obtain our main results, 
Theorems \ref{thm:limit} and \ref{thm:change-main} below, 
we state a fundamental integration formula
for the  function space $C_{a,b}[0,T]$.

\par
Let $\{e_1,\ldots,e_n\}$ be an orthonormal set in 
$(C_{a,b}'[0,T],\|\cdot\|_{C_{a,b}'})$,
let $k:\mathbb R^n \to \mathbb C$ be a Borel measurable function
and let $K:C_{a,b}[0,T]\to\mathbb C$ be given by equation 
\[
K(x)= k((e_1,x)^{\sim},\ldots,(e_n,x)^{\sim}).
\]
Then  
\begin{equation}\label{eq:c-formula}
\begin{split}
  E[K]
& =\int_{C_{a,b}[0,T]} k((e_1,x)^{\sim},\ldots,(e_n,x)^{\sim})d\mu(x)\\   
& =(2\pi)^{-n/2}  \int_{\mathbb R^n} k(u_1,\ldots,u_n)\\
& \qquad\ \quad   \times
\exp \bigg\{-\sum_{j=1}^{n}\frac{ [u_j-(e_j,a)_{C_{a,b}'}]^2 }{2}\bigg\} 
     du_1\cdots du_n  
\end{split}
\end{equation}
in the sense that if either side of equation \eqref{eq:c-formula} exists, 
both sides exist  and equality holds.

\par
We also need the following lemma to obtain 
our main  theorem in this section.

\begin{lemma}\label{lemma:limit}
Let $\{e_{1},\ldots,e_{n} \}$
be an orthonormal subset of $C_{a,b}'[0,T]$ 
and  let $w \in C_{a,b}'[0,T]$.
Then for each  $\lambda \in \mathbb C_+$, 
the function space integral 
\[
E_x[G_n(\lambda,x)\exp\{i(w,x)^{\sim}\}]
\]
exists and is given by the formula
\begin{equation}\label{eq:basic2}
\begin{split}
&  E_x[G_n(\lambda,x)\exp\{i(w,x)^{\sim}\}]\\
& =  \lambda^{-n/2} \exp \bigg\{
   \bigg[ \frac{\lambda-1}{2\lambda} \bigg] 
   \sum_{k=1}^{n}(e_k,w)_{C_{a,b}'}^{2}  
   - \frac{1}{2}\|w\|_{C_{a,b}'}^{2}  \\
&\qquad \qquad\qquad
+i\lambda^{-1/2}\sum_{k=1}^{n}(e_k,a)_{C_{a,b}'}(e_k,w)_{C_{a,b}'} \\
& \qquad\qquad\qquad
+ i(e_{n+1},a)_{C_{a,b}'} 
\bigg[ \|w\|_{C_{a,b}'}^{2}
   - \sum_{k=1}^{n}(e_k,w)_{C_{a,b}'}^{2}  \bigg]^{1/2}  \bigg\} 
\end{split}
\end{equation}
where $G_n$ is given by equation \eqref{eq:gn} above.
\end{lemma}
\begin{proof} (Outline)
Using the Gram-Schmidt process, for any $w\in C_{a,b}'[0,T]$ 
we can write  $w = \sum_{k=1}^{n+1}c_k e_k$   where
$\{e_1,\ldots, e_n, e_{n+1}\}$ is an orthonormal set in $C_{a,b}'[0,T]$
and
\[
c_k
 = \begin{cases} 
   (e_k,w)_{C_{a,b}'}     &,\quad   k=1,\ldots,n  \\    
  \big[ \|w\|_{C_{a,b}'}^{2}
        - \sum_{j=1}^{n}(e_j,w)_{C_{a,b}'}^{2}  \big]^{1/2}   
                          &,\quad k=n+1 
\end{cases}.
\]
Then using  \eqref{eq:gn},  \eqref{eq:c-formula}, 
the Fubini theorem   and  \eqref{eq:int-formula}, 
it follows that equation \eqref{eq:basic2} holds 
for all $\lambda>0$.
Finally \eqref{eq:basic2} holds for all $\lambda\in \mathbb C_+$ 
by analytic continuation.
\end{proof}

\par
The following remark will be very useful in the proof 
of our main theorem in this section.

\begin{remark}\label{re:domine}
Let $q_0$ be a positive real number 
and let $\Gamma_{q_0}$ be given by equation \eqref{eq:domain}.
For real numbers  $q_1$ and $q_2$ with $|q_j|>q_0$, $j\in\{1,2\}$,
let $\{\vec\lambda_n\}_{n=1}^{\infty}
    =\{(\lambda_{1,n},\lambda_{2,n})\}_{n=1}^{\infty}$
be a sequence in $\mathbb C_+^2$
such that 
\[
\vec\lambda_n=(\lambda_{1,n},\lambda_{2,n}) 
\to -i\vec q=(-iq_1,-iq_2).
\]
Let $\lambda_{j,n}=\alpha_{j,n}+i\beta_{j,n}$ 
for $j\in\{1,2\}$ and $n\in \mathbb N$.
Then for $j\in\{1,2\}$, 
$\text{\rm  Re}(\lambda_{j,n})=\alpha_{j,n} >0$  
and
\[
\lambda_{j,n}^{-1}
=(\alpha_{j,n}+i\beta_{j,n})^{-1}
=\frac{ \alpha_{j,n} -i\beta_{j,n}}{ \alpha_{j,n}^2+\beta_{j,n}^2} 
\]
for each $n\in \mathbb N$.
Since $|\hbox{\rm Im}((-iq_j)^{-1/2})|
         =1/\sqrt{2|q_j|}<1/\sqrt{2q_0}$ for $j\in\{1,2\}$,
there exists a sufficiently large $L\in \mathbb N$ such that 
for any $n\ge L$, $\lambda_{1,n}$ and  $\lambda_{2,n}$ 
are in $\hbox{\rm Int}(\Gamma_{q_0})$ and
\[
\begin{split}
\delta(q_1,q_2)
&\equiv
\sup \Big( \big\{ | \hbox{\rm Im} (\lambda_{1,n}^{-1/2})| 
                            : n\ge L \big\} \\
& \qquad \qquad \cup 
   \big\{ | \hbox{\rm Im} (\lambda_{2,n}^{-1/2})| : n\ge L \big\} \\
& \qquad \qquad \cup 
   \big\{|\hbox{\rm Im}((-iq_1)^{-1/2})|, 
            |\hbox{\rm Im}((-iq_2)^{-1/2})|\big\}\Big)
<\frac{1}{\sqrt{2q_0}}.
\end{split}
\]   
Thus there exists a positive real number $\varepsilon>1$ 
such that 
\[
\delta(q_1,q_2) < \frac{1}{\varepsilon}  \frac{1}{\sqrt{2 q_0}}. 
\]

\par
Let $\{e_n\}_{n=1}^{\infty}$ be a complete orthonormal set 
in $C_{a,b}'[0,T]$.
Using Parseval's identity, it follows that
\[
(g_1,g_2)_{C_{a,b}'}
=\sum\limits_{n=1}^{\infty}(e_n,g_1)_{C_{a,b}'}(e_n,g_2)_{C_{a,b}'}
\]
for all $g_1, g_2 \in C_{a,b}'[0,T]$.
In addition for   each $g\in C_{a,b}'[0,T]$,
\[
\|g\|_{C_{a,b}}^2-\sum\limits_{k=1}^{n}(e_k,g)_{C_{a,b}'}^2 
=\sum\limits_{k=n+1}^{\infty}(e_k,g)_{C_{a,b}'}^2  
\ge0
\]
for every $n\in \mathbb N$.
Note that for   $g\in C_{a,b}'[0,T]$, 
$(g,a)_{C_{a,b}'}$ may be positive, negative or zero.
Since 
\[
(g,a)_{C_{a,b}'}
=\sum\limits_{n=1}^{\infty}(e_n,g)_{C_{a,b}'}(e_n,a)_{C_{a,b}'}
\]
and for  $\varepsilon>1$,
\[
\begin{split}
-\varepsilon\|g\|_{C_{a,b}'}\|a\|_{C_{a,b}'}
&< -\|g\|_{C_{a,b}'}\|a\|_{C_{a,b}'}\\
&\le (g,a)_{C_{a,b}'}\\
&\le \|g\|_{C_{a,b}'}\|a\|_{C_{a,b}'}
<\varepsilon\|g\|_{C_{a,b}'}\|a\|_{C_{a,b}'},
\end{split}
\]
there exists a sufficiently large $K_j\in \mathbb N$
such that for any $n\ge K_j$,
\[
\begin{split}
\bigg|\sum\limits_{k=1}^{n}(e_k,A_j^{1/2}w)_{C_{a,b}'}
        (e_k,a  )_{C_{a,b}'}\bigg|
&< \varepsilon\|A_j^{1/2}w\|_{C_{a,b}'}\|a\|_{C_{a,b}'}\\
&\le \varepsilon\|A_j^{1/2}\|_{o}\|w\|_{C_{a,b}'}\|a\|_{C_{a,b}'}
\end{split}
\]
for  $j\in \{1,2\}$.

\par
Using these and a long and tedious calculation 
we can show  that for  every $n \ge \max\{L,K_1,K_2\}$,
\begin{equation*}\label{eq:re-domi}
\begin{split}
&\bigg|\exp\bigg\{\sum_{j=1}^2 
   \bigg(\bigg[ \frac{\lambda_{j,n}-1}{2\lambda_{j,n}} \bigg] 
\sum_{k=1}^{n}(e_k,A_j^{1/2}w)_{C_{a,b}'}^{2}  
- \frac{1}{2}\|A_j^{1/2}w\|_{C_{a,b}'}^{2}  \\
&\qquad \qquad\qquad
+i\lambda_{j,n}^{-1/2}\sum_{k=1}^{n} 
   (e_k,A_j^{1/2}w)_{C_{a,b}'}(e_k,a)_{C_{a,b}'} \\
& \qquad\qquad\qquad
+ i(e_{n+1},a)_{C_{a,b}'} 
\bigg[ \|A_j^{1/2}w\|_{C_{a,b}'}^{2}
   - \sum_{k=1}^{n}(e_k,A_j^{1/2}w)_{C_{a,b}'}^{2}  
    \bigg]^{1/2} \bigg) \bigg\}\bigg|\\
&<k(q_0; \vec A;w)
\end{split}
\end{equation*}
where $k(q_0; \vec A; w)$ is given by \eqref{eq:Domi-1}.
\end{remark}

\par
In our next theorem, for  $F \in \mathcal F_{A_1,A_2}^{\,\,a,b}$, 
we express the GFFT of $F$ 
as the limit of a sequence of function space integrals
on $C_{a,b}^2[0,T]$.

\begin{theorem}\label{thm:limit}
Let $q_0$ and $F$ be as in Theorem \ref{thm:tpq}.
Let $\{e_n\}_{n= 1}^{\infty}$ be a complete orthonormal set 
in $C_{a,b}'[0,T]$ 
and let $\{(\lambda_{1,n},\lambda_{2,n}) \}_{n=1}^{\infty} $ 
be a sequence in  $\mathbb C_+^2$ 
such that $\lambda_{j,n}\to -iq_j$
where $q_j$ is a real number with $|q_j|> q_0 $, $j\in\{1,2\}$. 
Then for $p\in[1,2]$ and for s-a.e. $(y_1,y_2)\in C_{a,b}^2[0,T]$, 
\[
\begin{split}
&T_{\vec q}^{(p)}(F)(y_1,y_2)\\
&=\lim_{n \to \infty} \lambda_{1,n}^{n/2} \lambda_{2,n}^{n/2} 
   E_{\vec x}[G_n(\lambda_{1,n} ,x_1)G_n(\lambda_{2,n} ,x_2)    
   F(y_1+x_1,y_2+x_2) ]
\end{split}
\]
where $G_n$ is given by equation  \eqref{eq:gn} above.
\end{theorem}
\begin{proof}
By Theorems \ref{thm:t1q} and \ref{thm:tpq} above, 
we know that for each $p\in[1,2]$, 
the $L_p$ analytic GFFT  of $F$, $T_{\vec q}^{(p)}(F)$  
exists and is given by the right hand side 
of equation    \eqref{eq:tpq}.
Thus it suffices to show that
\[
\begin{split}
&T_{\vec q}^{(1)}(F)(y_1,y_2)
 =E_{\vec x}^{\text{\rm anf}_{\vec q}}[F(y_1+x_1,y_2+x_2)]\\
&=  \lim_{n \to \infty} \lambda_{1,n}^{n/2}  \lambda_{2,n}^{n/2} 
E_{\vec x}[G_n(\lambda_{1,n} ,x_1)G_n(\lambda_{2,n} ,x_2)    
     F(y_1+x_1,y_2+x_2) ].
\end{split}
\]

\par
Using equation \eqref{eq:element2}, the Fubini theorem 
and equation  \eqref{eq:basic2}
with $\lambda$ and $w$ replaced with $\lambda_{j,n}$ and $A_j^{1/2}w$, 
$j\in\{1,2\}$, respectively,
we see that 
\begin{equation}\label{eq:limit00}
\begin{split}
&\lambda_{1,n}^{n/2}\lambda_{2,n}^{n/2}
E_{\vec x} [ G_n(\lambda_{1,n} ,x_1)G_n(\lambda_{2,n} ,x_2)
   F(y_1+x_1,y_2+x_2) ] \\
&= \lambda_{1,n}^{n/2}\lambda_{2,n}^{n/2}
\int_{C_{a,b}'[0,T]} \exp\bigg\{
  \sum_{j=1}^2i(A_j^{1/2}w,y_j)^{\sim}\bigg\} \\
&\qquad\qquad\quad\times 
\bigg(\prod_{j=1}^2E_{x_j} \big[G_n(\lambda_{1,n}^{-1/2},x_j)
\exp \big\{i(A_j^{1/2}w,x_j)^{\sim} \big\}  \big] \bigg) df(w)   \\
&=   \int_{C_{a,b}'[0,T]}   
 \exp \bigg\{ \sum_{j=1}^2\bigg(  i(A_j^{1/2}w,y_j)^{\sim}
 + \bigg[ \frac{\lambda_{j,n} -1}{2\lambda_{j,n}} \bigg]
\sum_{k=1}^{n}(e_k,A_j^{1/2}w)_{C_{a,b}'}^{2}\\
&\qquad\qquad
-\frac{1}{2} \|A_j^{1/2}w\|_{C_{a,b}'}^{2}  
+ i\lambda_{j,n}^{-1/2}
\sum _{k=1}^{n} (e_k,a)_{C_{a,b}'}(e_k,A_j^{1/2}w)_{C_{a,b}'} \\
& \qquad\qquad
+i(e_{n+1},a)_{C_{a,b}'}\bigg[ \|A_j^{1/2}w\|_{C_{a,b}'}^{2} 
   - \sum_{k=1}^{n}(e_k,A_j^{1/2}w)_{C_{a,b}'}^{2}   
 \bigg ]^{1/2}   \bigg)   \bigg\}  df(w) .
\end{split}
\end{equation}
But, by Remark \ref{re:domine} we see that 
the last expression of  \eqref{eq:limit00} 
is dominated by   \eqref{eq:finite000} on the region $\Gamma_{q_0}$ 
given by equation \eqref{eq:domain}   
for all but a finite number of values of   $n$.
Next  using the dominated convergence theorem,   
Parseval's relation and equation \eqref{eq:tpq}, 
it follows that for s-a.e. $(y_1,y_2)\in C_{a,b}^2[0,T]$,
\[
\begin{split}
&\lim_{n \to \infty}  \lambda_{1,n}^{ n/2}\lambda_{2,n}^{ n/2}
E_{\vec x} [ G_n(\lambda_{1,n} ,x_1)G_n(\lambda_{2,n},x_2)   
    F(y_1+x_1,y_2+x_2) ]      \\
&=\int_{C_{a,b}'[0,T]} \lim_{n \to \infty}
\exp \bigg\{ \sum_{j=1}^2\bigg(
  i(A_j^{1/2}w,y_j)^{\sim}
 + \bigg[ \frac{\lambda_{j,n} -1}{2\lambda_{j,n}} \bigg]
\sum_{k=1}^{n}(e_k,A_j^{1/2}w)_{C_{a,b}'}^{2}\\
&\qquad\qquad\qquad\qquad
-\frac{1}{2} \|A_j^{1/2}w\|_{C_{a,b}'}^{2}  
+ i\lambda_{j,n}^{-1/2}
\sum _{k=1}^{n} (e_k,a)_{C_{a,b}'}(e_k,A_j^{1/2}w)_{C_{a,b}'} \\
& \qquad\qquad
+i(e_{n+1},a)_{C_{a,b}'}   \bigg[ \|A_j^{1/2}w\|_{C_{a,b}'}^{2} 
  - \sum_{k=1}^{n}(e_k,A_j^{1/2}w)_{C_{a,b}'}^{2} \bigg ]^{1/2} \bigg)\bigg\}
  df(w) \\
&= \int_{C_{a,b}'[0,T]} \exp \bigg\{\sum_{j=1}^2i(A_j^{1/2}w,y_j)^{\sim}\bigg\}
    \psi(-i\vec q;\vec A;w)   df(w)\\
&=T_{\vec q}^{(1)}(F)(y_1,y_2) 
\end{split}
\]
which concludes the proof of Theorem \ref{thm:limit}.
\end{proof}

\begin{corollary} 
Let $q_0$, $F$, $\{e_n\}_{n= 1}^{\infty}$,  
$\{(\lambda_{1,n},\lambda_{2,n})\}_{n=1}^{\infty}$ 
and $(q_1,q_2)$
be as in Theorem \ref{thm:limit}. 
Then
\begin{equation*}\label{eq;gfi-limit}
E_{\vec x}^{\text{\rm anf}_{\vec q}}[F(x_1,x_2) ]
=\lim_{n \to \infty} \lambda_{1,n}^{n/2} \lambda_{2,n}^{n/2} 
E_{\vec x}[G_{n}(\lambda_{1,n} ,x_1)G_{n}(\lambda_{2,n} ,x_2) 
    F(x_1,x_2) ]
\end{equation*}
where $G_n$ is given by equation  \eqref{eq:gn} above.
\end{corollary}

\begin{corollary}\label{thm:limit2}
Let $q_0$, $F$ and  $\{e_n\}_{n= 1}^{\infty}$  
be as in Theorem \ref{thm:limit}
and let $\Gamma_{q_0}$ be given by \eqref{eq:domain}.
Let $\vec\lambda=(\lambda_1,\lambda_2) 
\in \hbox{\rm Int}(\Gamma_{q_0})$  
and let $\{(\lambda_{1,n},\lambda_{2,n}) \}_{n=1}^{\infty}$ 
be a sequence in  $\mathbb C_+^2$ such that 
$\lambda_{j,n}\to \lambda_j $, $j\in\{1,2\}$. 
Then
\begin{equation}\label{eq:limit02}
E_{\vec x}^{\text{\rm an}_{\vec \lambda}}
 [F(x_1,x_2)]=\lim_{n \to \infty} 
\lambda_{1,n}^{n/2} \lambda_{2,n}^{n/2} 
E_{\vec x}[ G_n(\lambda_{1,n} ,x_1)G_n(\lambda_{2,n} ,x_2) F(x_1,x_2)]
\end{equation}
where $G_n$ is given by equation \eqref{eq:gn} above.
\end{corollary}

\par
Our another result, namely a change of scale formula 
for function space integrals now 
follows   from Corollary \ref{thm:limit2} above.

\begin{theorem}\label{thm:change-main}
Let $F \in \mathcal F_{A_1,A_2}^{\,\,a,b}$
and let $\{e_n\}_{n=1}^{\infty}$ 
be a complete orthonormal set in $C_{a,b}'[0,T]$.
Then for any $\rho_1 > 0$ and $\rho_2>0$, 
\[
E_{\vec x}[F(\rho_1 x_1,\rho_2 x_2)]
= \lim_{n \to \infty} \rho_1^{-n}  \rho_2^{-n}  
E_{\vec x} [ G_n(\rho_1^{-2},x_1)G_n(\rho_2^{-2},x_2) F(x_1,x_2)]
\]
where $G_n$ is given by equation \eqref{eq:gn} above.
\end{theorem}
\begin{proof}
 Simply choose  $\lambda_{j} = \rho_j^{-2}$ for $j\in\{1,2\}$ 
and  $ \lambda_{j,n} = \rho_j^{-2}$ for $j\in\{1,2\}$  
and  $n\in\mathbb N$ in equation \eqref{eq:limit02}.
\end{proof}

\begin{remark}
Of course, if we choose $a(t)\equiv 0$, $b(t)=t$, 
$A_1=I$(identity operator)  and $A_2=0$(zero operator),
then the function space $C_{a,b}[0,T]$ reduces 
to the classical Wiener space $C_0[0,T]$
and the generalized Fresnel type class $\mathcal{F}_{A_1,A_2}^{\,\,a,b}$
reduces to the Fresnel class $\mathcal{F}(C_0[0,T])$.
It is known that $\mathcal{F}(C_0[0,T])$ 
forms a Banach algebra over the complex field and that
$\mathcal{F}(C_0[0,T])$ and $\mathcal{S}$
are isometrically isomorphic. 
See \cite{Johnson82}.
In this case, we have the relationships 
between  the analytic Feynman integral and the Wiener integral
on classical Wiener space  as discussed in \cite{CS87I} and \cite{CS87II}.
\end{remark}

\setcounter{equation}{0}
\section{The first variation of functionals in 
$\mathcal{F}_{A_1,A_2}^{\,\,a,b}$}\label{sec:5}

In this section, 
we first  give the definition of 
the first variation of a functional $F$ on $C_{a,b}^2[0,T]$.
The following definition
of the first variation on product space 
is due to Yoo and Kim \cite{YK07}.

\begin{definition}
Let $F$ be a  functional on $C_{a,b}^2[0,T]$
and let $g_1$ and $g_2$ be elements of  $C_{a,b}[0,T]$.
Then
\begin{equation}\label{eq:1st}
\delta F(x_1,x_2|g_1,g_2)
 =\frac{\partial}{\partial h}F(x_1+h g_1,x_2)\bigg|_{h=0}
+\frac{\partial}{\partial h}F(x_1, x_2+h g_2)\bigg|_{h=0}
\end{equation}
(if it exists) is called the first variation of $F$ 
in the direction of $(g_1, g_2)$.
\end{definition} 

\par
Throughout this section, 
when working with $\delta F(x_1, x_2|g_1,g_2)$,
we will always require $g_1$ and $g_2$ 
to be   elements of $C_{a,b}'[0,T]$.

\par
For $j\in\{1,2\}$, let $g_j\in C_{a,b}'[0,T]$
and let $F$ be an element of $\mathcal{F}_{A_1,A_2}^{\,\,a,b}$
whose associated measure $f$, see equation \eqref{eq:element2}, 
satisfies the inequality
\begin{equation}\label{eq:finite102}
\int_{C_{a,b}'[0,T]}\|w\|_{C_{a,b}'} d|f|(w)<+\infty.
\end{equation}
Then using equation \eqref{eq:1st}, 
we obtain that 
\begin{equation}\label{eq:evaluation-delta}
\begin{split}
&\delta F(x_1,x_2|g_1,g_2)\\
&=\sum_{k=1}^2\bigg[\frac{\partial}{\partial h}\bigg(
 \int_{C_{a,b}'[0,T]}\exp\bigg\{\sum_{j=1}^2i(A_j^{1/2}w,x_j)^{\sim}\\
&\qquad\qquad\qquad\qquad\qquad\qquad\qquad
+ ih(A_k^{1/2}w,g_k)^{\sim}\bigg\} d f(w)\bigg)\bigg|_{h=0}\bigg]\\ 
&= \int_{C_{a,b}'[0,T]}\Big[\sum_{k=1}^2i(A_k^{1/2}w,g_k)_{C_{a,b}'}\Big]
\exp\bigg\{\sum_{j=1}^2i(A_j^{1/2}w,x_j)^{\sim}\bigg\} df(w)\\
&= \int_{C_{a,b}'[0,T]}\exp\bigg\{\sum_{j=1}^2i(A_j^{1/2}w,x_j)^{\sim}\bigg\} 
d \sigma^{\vec A,\vec g}(w)
\end{split}
\end{equation}
where the complex measure $\sigma^{\vec A,\vec g}$ 
is defined by
\[
\sigma^{\vec A,\vec g}(B)
= \int_{B}\Big[\sum_{k=1}^2 i(A_{k}^{1/2}w,g_k)_{C_{a,b}'} \Big]df(w), 
 \quad B \in \mathcal{B}(C_{a,b}'[0,T]).
\]
The second equality of \eqref{eq:evaluation-delta}
follows from \eqref{eq:finite102} 
and Theorem 2.27 in \cite{Folland}.
Also, $\delta F(x_1,x_2|g_1,g_2)$ is an element 
of $\mathcal{F}_{A_1,A_2}^{\,\,a,b}$
as a functional of $(x_1,x_2)$,
since by the Cauchy-Schwartz inequality and \eqref{eq:finite102},
\[
\begin{split}
\|\sigma^{\vec A,\vec g}\|
&\le \int_{C_{a,b}'[0,T]}\sum_{j=1}^2|i(A_j^{1/2}w,g_j)_{C_{a,b}'}|d|f|(w)\\
&\le \int_{C_{a,b}'[0,T]}\sum_{j=1}^2
   \|A_j^{1/2}\|_{o}\|w\|_{C_{a,b}'}\|g_j\|_{C_{a,b}'}d|f|(w)\\
&\le\bigg(\sum_{j=1}^2\|A_j^{1/2}\|_{o}\|g_j\|_{C_{a,b}'}\bigg)
   \int_{C_{a,b}'[0,T]} \|w\|_{C_{a,b}'} d|f|(w) <+\infty,
\end{split}
\]
where $\|A_j^{1/2}\|_o$ is the operator norm of $A_j^{1/2}$.

\par
For given positive real number $q_0$, 
we  define a subclass $\mathcal{G}_{A_1,A_2}^{\, \,q_0}$ 
of  $\mathcal{F}_{A_1,A_2}^{\,\,a,b}$ by
$F\in \mathcal{G}_{A_1,A_2}^{\,\, q_0}$  
if and only if 
\[
\int_{C_{a,b}'[0,T]}\|w\|_{C_{a,b}'}
k(q_0;\vec A;w)d|f|(w)<+\infty   
\]
where $f$, the associated measure of $F$, 
and $F$ are  related by equation \eqref{eq:element2}
and $k(q_0; \vec A; w)$ is given by equation \eqref{eq:Domi-1}.

\par
Our next two theorems follows quite readily 
from  the techniques developed in 
Sections \ref{sec:3} and \ref{sec:4} above.

\begin{theorem}\label{thm:limit-1st}
Let $q_{0}$ be a positive real number and 
let $g_1$ and $g_2$ be elements of $C_{a,b}'[0,T]$.
Let $F$ be an element of $\mathcal{G}_{A_1,A_2}^{\, \,q_0}$
and let $\Gamma_{q_0}$ be given by \eqref{eq:domain}.
Then: 

{\rm (1)}
for all real numbers $q_1$ and $q_2$ with $|q_j|>  q_{0}$, 
$j\in\{1,2\}$  and all $p\in[1,2]$,
the $L_p$   analytic GFFT of $\delta F(\,\cdot\,,\cdot\,|g_1,g_2)$,
exists, is an element of $\mathcal{F}_{A_1,A_2}^{\,\,a,b}$ 
and is given by the formula
\[
\begin{split}
&T_{\vec q}^{(p)}(\delta F(\,\cdot\,,\cdot\,|g_1,g_2))(y_1,y_2) \\
&= \int_{C_{a,b}'[0,T]}\Big[\sum_{j=1}^2i(A_{j}^{1/2}w,g_j)_{C_{a,b}'}\Big]
 \exp\bigg\{\sum_{j=1}^2i(A_{j}^{1/2}w,y_j)^{\sim}\bigg\}
   \psi(-i\vec q;\vec A;w)df(w)
\end{split}
\]
for s-a.e. $(y_1,y_2) \in C_{a,b}^2[0,T]$,
where $\psi$ is given by equation \eqref{eq:h-function};
and

{\rm (2)}
for all real numbers $q_1$ and $q_2$ with $|q_j|> q_{0}$, $j\in\{1,2\}$,
the generalized analytic Feynman integral   
of $\delta F( \cdot ,\cdot |g_1,g_2)$ exists 
and  is given by the formula
\begin{equation}\label{eq:gfi-delta}
\begin{split}
&E_{\vec x}^{\text{\rm anf}_{\vec q}} [\delta F(x_1,x_2|g_1,g_2)]\\
&\qquad = \int_{C_{a,b}'[0,T]}\Big[\sum_{j=1}^2i(A_{j}^{1/2}w,g_j)_{C_{a,b}'}\Big]
 \psi(-i\vec q;\vec A;w)df(w).
\end{split}
\end{equation}
In addition, 
for each $\vec\lambda \in \hbox{\rm Int}(\Gamma_{q_0})$,
\[
\begin{split}
&E_{\vec x}^{\text{\rm an}_{\vec\lambda}}[\delta F(x_1,x_2|g_1,g_2)] \\
&\qquad
 = \int_{C_{a,b}'[0,T]}\Big[\sum_{j=1}^2i(A_{j}^{1/2}w,g_j)_{C_{a,b}'}\Big]
    \psi(\vec\lambda;\vec A;w)df(w).
\end{split}
\]
\end{theorem}

\begin{theorem}
Let $q_0$, $\{e_n\}_{n= 1}^{\infty}$,  
$\{(\lambda_{1,n},\lambda_{2,n})\}_{n=1}^{\infty}$ 
and $(q_1,q_2)$ be as in Theorem \ref{thm:limit} above, 
and let $g_1$ and $g_2$ be elements of $C_{a,b}'[0,T]$.
Let $F$ be an element of $\mathcal G_{A_1,A_2}^{\,\,q_0}$
and let $\Gamma_{q_0}$ be given by \eqref{eq:domain}.
Then:

{\rm (1)}
for  all $p\in[1,2]$,
\[
\begin{split}
&T_{\vec q}^{(p)}(\delta F(\,\cdot\,,\cdot\,| g_1,g_2))(y_1,y_2)\\
&=\lim_{n \to \infty} \lambda_{1,n}^{n/2} \lambda_{2,n}^{n/2} 
E_{\vec x} [G_n(\lambda_{1,n} ,x_1)G_n(\lambda_{2,n} ,x_2)   
 \delta F(y_1+x_1,y_2+x_2|g_1,g_2) ]
\end{split}
 \]
for s-a.e. $(y_1,y_2)\in C_{a,b}^2[0,T]$, 
where $G_n$ is given by equation  \eqref{eq:gn} above;

{\rm (2)}
the generalized analytic Feynman integral 
$E^{\text{\rm anf}_{\vec q}}[\delta F(\cdot,\cdot|g_1,g_2)]$
of $\delta F( \cdot ,\cdot |g_1,g_2)$ 
is expressed as follows:
\[
\begin{split}
&E_{\vec x}^{\text{\rm anf}_{\vec q}}[\delta F(x_1,x_2|g_1,g_2)]\\
&=\lim_{n \to \infty} \lambda_{1,n}^{n/2} \lambda_{2,n}^{n/2} 
E_{\vec x} [G_n(\lambda_{1,n} ,x_1)G_n(\lambda_{2,n} ,x_2)   
    \delta  F(x_1,x_2|g_1,g_2) ].
\end{split}
 \]
Also  for each $\vec\lambda \in \hbox{\rm Int}(\Gamma_{q_0})$
and all sequence $\{(\lambda_{1,n},\lambda_{2,n})\}_{n=1}^{\infty}$
in $\mathbb C_+^2$ which converges to $\vec \lambda$,
\begin{equation*} 
\begin{split}
&E_{\vec x}^{\text{\rm an}_{\vec \lambda}}
   [\delta F(x_1,x_2|g_1,g_2) ]\\
&=\lim_{n \to \infty} 
\lambda_{1,n}^{n/2} \lambda_{2,n}^{n/2} 
E_{\vec x} [ G_n(\lambda_{1,n} ,x_1)G_n(\lambda_{2,n} ,x_2) 
\delta F(x_1,x_2|g_1,g_2)],
\end{split}
\end{equation*}
and

{\rm (3)}
for any $\rho_1 > 0$ and $\rho_2>0$, 
\begin{equation*}
E_{\vec x} [\delta F(\rho_1 x_1,\rho_2 x_2|g_1,g_2)]
= \lim_{n \to \infty} \rho_1^{-n}  \rho_2^{-n}  
E_{\vec x}[ G_n(\rho_1^{-2},x_1)G_n(\rho_2^{-2},x_2) F(x_1,x_2)].
\end{equation*}
\end{theorem}

\setcounter{equation}{0}
\section{Functionals in   $\mathcal{F}_{A_1,A_2}^{\,\,a,b}$}\label{sec:6}

In this section, we first prove  a theorem 
ensuring that various functionals 
are in $\mathcal{F}_{A_1,A_2}^{\,\,a,b}$.

\begin{theorem}\label{thm:potential}
Let $A_1$ and $A_2$ be bounded, nonnegative,  
self-adjoint operators on $C_{a,b}'[0,T]$.
Let $(Y, \mathcal Y, \gamma)$ be a $\sigma$-finite measure space
and let
$\varphi_l: Y\to C_{a,b}'[0,T]$ be 
$\mathcal Y$--$\mathcal{B}(C_{a,b}'[0,T])$
measurable for $l\in\{1,\ldots, d\}$.
Let $\theta:  Y\times \mathbb R^d\to\mathbb C$ be given by
$\theta (\eta;\cdot)=\widehat \nu_{\eta} (\cdot)$
where $\nu_{\eta}\in \mathcal{M}(\mathbb R^d)$ for every $\eta\in Y$
and where the family
$\{\nu_{\eta}:\eta\in Y\}$ satisfies:
\begin{enumerate}
   \item[(i)]
$\nu_\eta(E)$ is a $\mathcal{Y}$-measurable function 
of $\eta$ for every $E\in \mathcal{B}(\mathbb R^d)$; and
   \item[(ii)] 
$\|\nu_{\eta}\| \in L^1(Y,\mathcal Y, \gamma)$.
\end{enumerate}

\par
Under these hypothesis, the functional $F: C_{a,b}^2[0,T]\to\mathbb C$
given by 
\begin{equation}\label{eq:A-function}
F(x_1,x_2)
=\int_{Y}\theta\bigg(\eta; \sum_{j=1}^2  (A_j^{1/2}\varphi_1(\eta),x_j)^{\sim},
 \ldots, \sum_{j=1}^2  (A_j^{1/2}\varphi_d(\eta),x_j)^{\sim}\bigg)d\gamma(\eta)
\end{equation}
belongs to $\mathcal{F}_{A_1,A_2}^{\,\,a,b}$ 
and satisfies the inequality
\begin{equation*}
\|F\|\le \int_{Y}\|\nu_{\eta}\| d \gamma(\eta).
\end{equation*}
\end{theorem}
\begin{proof}
Using the techniques similar to those used in  \cite{CJS87},
we can show that $\|\nu_{\eta}\|$ is measurable as a function of $\eta$,
that $\theta$ is $\mathcal{Y}$-measurable, 
and that the integrand in equation 
\eqref{eq:A-function} is a measurable function of $\eta$
for every $(x_1,x_2) \in C_{a,b}^2[0,T]$.

\par
We  define a measure $\tau$ 
on $\mathcal Y\times \mathcal{B}(\mathbb R^d)$ by
\[
\tau(E)=\int_{Y}\nu_{\eta}(E^{(\eta)})d\gamma(\eta),
\quad\hbox{for}\,\, E\in \mathcal Y\times\mathcal{B}(\mathbb R^d).
\]
Then by  the first assertion of Theorem 3.1 in \cite{JS83}, 
$\tau$ satisfies $\|\tau\|\le\int_Y\|\nu_{\eta}\|d \gamma(\eta)$.
Now let
$\Phi:Y\times\mathbb R^d\to C_{a,b}'[0,T]$ be defined by 
$\Phi(\eta; v_1,\ldots, v_d)=\sum_{l=1}^dv_l\varphi_l(\eta)$.
Then $\Phi$ is 
$\mathcal{Y}\times\mathcal{B}(\mathbb R^d)$--$\mathcal{B}(C_{a,b}'[0,T])$-measurable
on the hypothesis for $\varphi_l$, $l\in\{1,\ldots,d\}$.
Let   $\sigma=\tau\circ \Phi^{-1}$.
Then clearly $\sigma \in \mathcal{M}(C_{a,b}'[0,T])$ 
and satisfies $\|\sigma\|\le\|\tau\|$.

\par
From the  change of variables theorem 
and the second assertion  of Theorem 3.1 in \cite{JS83},
it follows that
for a.e. $(x_1,x_2)\in C_{a,b}^2[0,T]$ 
and for every $\rho_1>0$ and $\rho_2>0$,
\[
\begin{split}
&F(\rho_1x_1,\rho_2x_2)\\
&=\int_{Y}\widehat\nu_{\eta}
  \bigg(  \sum_{j=1}^2  (A_j^{1/2}\varphi_1(\eta),\rho_j x_j)^{\sim},
 \ldots,  \sum_{j=1}^2  (A_j^{1/2}\varphi_d(\eta),\rho_j x_j)^{\sim}  
 \bigg)d\gamma(\eta)\\
&=\int_Y\bigg[\int_{\mathbb R^d} \exp\bigg\{ i\sum_{l=1}^d v_l
\bigg[ \sum_{j=1}^2  ( A_j^{1/2} \varphi_l (\eta), \rho_j x_j )^{\sim} \bigg]\bigg\}
 d\nu_{\eta}(v_1,\ldots,v_d)\bigg]d\gamma(\eta)\\
&=\int_{Y\times\mathbb R^d} \exp\bigg\{ i\sum_{l=1}^d v_l
\bigg[ \sum_{j=1}^2  ( A_j^{1/2} \varphi_l (\eta), \rho_j x_j )^{\sim} \bigg]\bigg\}
 d \tau(\eta; v_1,\ldots,v_d)\\
\end{split}
\]
\[
\begin{split}
&=\int_{Y\times\mathbb R^d}\exp\bigg\{ \sum_{j=1}^2 i(A_j^{1/2}
\Phi(\eta; v_1,\ldots,v_d), \rho_jx_j)^{\sim} \bigg\}
d \tau(\eta; v_1,\ldots,v_d)\\
&=\int_{C_{a,b}'[0,T]}\exp\bigg\{ \sum_{j=1}^2 
i(A_j^{1/2}w,\rho_jx_j)^{\sim}\bigg\}d\tau\circ\Phi^{-1}(w)\\
&=\int_{C_{a,b}'[0,T]}\exp\bigg\{ \sum_{j=1}^2 
i(A_j^{1/2}w,\rho_jx_j)^{\sim}\bigg\}d\sigma(w).
\end{split}
\]
Thus the functional  $F$  given by equation \eqref{eq:A-function} 
belongs to $\mathcal{F}_{A_1,A_2}^{\,\,a,b}$
and satisfies the inequality
\[
\|F\|=\|\sigma\|\le\|\tau\|\le\int_{Y}\|\nu_{\eta}\|d\gamma(\eta).
\]
\end{proof}

\par 
As mentioned in (2) of Remark \ref{re:A}, 
$\mathcal{F}_{A_1,A_2}^{\,\,a,b}$ is a Banach algebra
if $\text{\rm  Ran}(A_1+A_2)$ is dense in $C_{a,b}'[0,T]$.
In this case,  many analytic functionals of $F$ can be formed.
The following corollary is  relevant to Feynman integration theories
and quantum mechanics where exponential functions 
play an important role.

\begin{corollary}
Let $A_1$ and $A_2$ be  bounded, nonnegative and 
  self-adjoint operators on $C_{a,b}'[0,T]$
such that 
$\text{\rm  Ran}(A_1+A_2)$ is dense in $C_{a,b}'[0,T]$.
Let $F$ be given by equation \eqref{eq:A-function}
with $\theta$ as in Theorem \ref{thm:potential},
and let $\beta:\mathbb C\to \mathbb C$ be an entire function.
Then $(\beta\circ F)(x_1,x_2)$ is in $\mathcal{F}_{A_1,A_2}^{\,\,a,b}$.
In particular, $\exp\{F(x_1,x_2)\}\in \mathcal{F}_{A_1,A_2}^{\,\,a,b}$.
\end{corollary}

\begin{corollary}
Let $A_1$ and $A_2$ be  bounded, nonnegative, 
self-adjoint operators on $C_{a,b}'[0,T]$,
and let $\{g_1,\ldots, g_d\}$ be a finite subset of $C_{a,b}'[0,T]$.
Given $\beta=\widehat{\nu}$ where $\nu\in \mathcal{M}(\mathbb R^d)$, 
define  $F :C_{a,b}^2[0,T]\to\mathbb C$ by
\[
F(x_1,x_2)=\beta\bigg(\sum_{j=1}^2(A_j^{1/2}g_1,x_j)^{\sim},
  \ldots, \sum_{j=1}^2(A_j^{1/2}g_d,x_j)^{\sim}\bigg).
\]
Then $F $ is  an element of $\mathcal{F}_{A_1,A_2}^{\,\,a,b}$.
\end{corollary}
\begin{proof}
Let $(Y,\mathcal{Y},\gamma)$ be a probability space
and for $l\in\{1,\ldots, d\}$, let $\varphi_l(\eta)\equiv g_l$.
Take $\theta(\eta;\cdot)=\beta(\cdot)=\widehat\nu(\cdot)$.
Then for all $\rho_1>0$ and $\rho_2>0$ and 
for a.e. $(x_1,x_2)\in C_{a,b}^2[0,T]$,
\[
\begin{split}
&\int_Y \theta\bigg(\eta; \sum_{j=1}^2  
(A_j^{1/2}\varphi_1(\eta),\rho_jx_j)^{\sim},\ldots,
 \sum_{j=1}^2  (A_j^{1/2}\varphi_d(\eta),\rho_jx_j)^{\sim}\bigg)
d\gamma(\eta)\\
&=\int_Y \beta\bigg( \sum_{j=1}^2  
(A_j^{1/2}g_1 ,\rho_jx_j)^{\sim},\ldots,
 \sum_{j=1}^2  (A_j^{1/2}g_d,\rho_jx_j)^{\sim}\bigg)d\gamma(\eta)\\
&= \beta\bigg( \sum_{j=1}^2  (A_j^{1/2}g_1 ,\rho_jx_j)^{\sim},
\ldots,  \sum_{j=1}^2  (A_j^{1/2}g_d,\rho_jx_j)^{\sim}\bigg) \\
&=F (\rho_1x_1,\rho_2x_2).
\end{split}
\]
Hence $F \in \mathcal{F}_{A_1,A_2}^{\,\,a,b}$.
\end{proof}

\begin{remark}
Let $d=1$  and let 
$(Y,\mathcal{Y},\gamma)=([0,T],\mathcal{B}([0,T]), m_L)$ 
in Theorem \ref{thm:potential}
where $m_L$ denotes   Lebesgue measure on $[0,T]$.
Then Theorems 4.6, 4.7 and 4.9 in \cite{CCL09}    
follows from the results in this section
by letting $A_1$ be the identity operator 
and letting $A_2\equiv0$ on $C_{a,b}'[0,T]$.
The function $\theta$ studied  in \cite{CCL09} 
(and  mentioned in (1) of Remark \ref{re:physics} above) 
is interpreted as the potential energy in   quantum mechanics.
\end{remark}

\setcounter{equation}{0}
\section{A translation theorem for the generalized analytic Feynman integral
of functionals in $\mathcal{F}_{A_1,A_2}^{\,\,a,b}$}\label{sec:7}

In \cite{CS82}, Cameron and Storvick derived  a translation theorem 
for the analytic Feynman integral of functionals 
in the Banach algebra $\mathcal S$ on classical Wiener space
and in \cite{CC96}, 
Chang and Chung derived a translation theorem 
for  function space integral of functionals on $C_{a,b}[0,T]$.
The translation theorem in \cite{CC96}, 
using the notation of this paper, 
states that if $x_0\in C_{a,b}'[0,T]$
and if $G$ is a $\mu$-integrable function on $C_{a,b}[0,T]$,
then
\begin{equation}\label{eq:translation}
E[G(x+x_0)]
=\exp\bigg\{-\frac12\|x_0\|^2_{C_{a,b}'}
   -(x_0,a)_{C_{a,b}'}\bigg\}E[G(x)\exp\{(x_0,x)^{\sim}\}].
\end{equation}

\par
In this section,   
we will present a generalized analytic Feynman integral version 
of the translation theorem for functionals in
$\mathcal{F}_{A_1,A_2}^{\,\,a,b}$.

\begin{theorem}\label{thm:tpq-trans}
Let $q_0$ and $F$ be as in Theorem \ref{thm:t1q}.
Let  $g_1$ and $g_2$ be elements of $C_{a,b}'[0,T]$. 
Then  for all $p\in[1,2]$, all real numbers $q_1$ 
and $q_2$ with $|q_j|> q_0$, $j\in\{1,2\}$
 and for s-a.e. $(y_1, y_2)\in C_{a,b}^2[0,T]$, 
\begin{equation}\label{eq:tpq-trans}
\begin{split}
&T_{\vec q}^{(p)}(F)(y_1+A_1^{1/2}g_1,y_2+A_2^{1/2}g_2)\\
&= \exp\bigg\{\sum\limits_{j=1}^2
    \bigg[\frac{iq_j}{2}(A_jg_j,g_j)_{C_{a,b}'}
    -(-iq_j)^{1/2} (A_j^{1/2}g_j,a)_{C_{a,b}'}\bigg]\bigg\}\\
&\quad\times 
\exp\bigg\{\sum_{j=1}^2iq_j(A_j^{1/2}g_j,y_j)^{\sim}\bigg\}
    T_{\vec q}^{(p)}(F^*)(y_1,y_2)\\
\end{split}
\end{equation} 
where
\[
F^*(y_1,y_2)
= F(y_1,y_2)\exp\bigg\{\sum\limits_{j=1}^2
\Big[-iq_j(A_j^{1/2}g_j,y_j)^{\sim}\Big]\bigg\}.
\]
\end{theorem}
\begin{proof} 
By Theorems \ref{thm:t1q} and \ref{thm:tpq}, 
the $L_p$ analytic GFFT $T_{\vec q}^{(p)}(F)$ of $F$ 
exists for all $p\in [1,2]$ and 
is given by the right hand side of equation \eqref{eq:tpq}.
Thus we  only  need  to verify the equality 
in equation \eqref{eq:tpq-trans}.
We will give the proof for the case $p\in(1,2]$.
The case $p=1$ is similar, but somewhat easier.

\par
For $\lambda_j>0$, $j\in\{1,2\}$ and $w \in C_{a,b}'[0,T]$, 
let $G_j(w;\cdot)$ be a functional on $C_{a,b}[0,T]$ given by
\begin{equation}\label{eq:G1}
G_j(w;x_j )
=\exp\big\{i (A_j^{1/2}w,\lambda_j^{-1/2} x_j  )^{\sim}\big\}
\end{equation}
and  let 
\begin{equation}\label{eq:x00}
x_{0,1}=\lambda_1^{1/2}A_{1}^{1/2}g_1\,\, \hbox{  and } \,\,
x_{0,2}=\lambda_2^{1/2}A_{2}^{1/2}g_2.
\end{equation}
Then for $j\in\{1,2\}$,
\begin{equation}\label{eq:x01}
\|x_{0,j} \|_{C_{a,b}'}^2
=\lambda_j(A_j g_j,g_j)_{C_{a,b}'}
\,\,\hbox{ and }\,\, (
x_{0,j},a)_{C_{a,b}'}
=\lambda_j^{1/2}(A_j^{1/2}g_j,a)_{C_{a,b}'}.
\end{equation}

\par
Using   \eqref{eq:element2}, the Fubini theorem, 
\eqref{eq:x00}, \eqref{eq:G1}, \eqref{eq:translation} 
and \eqref{eq:x01},
we obtain that for $\lambda_1>0$ and $\lambda_2>0$,
\[
\begin{split}
&T_{\vec \lambda}(F)(y_1+A_1^{1/2}g_1,y_2+A_2^{1/2}g_2)\\
&=\int_{C_{a,b}'[0,T]}\exp\bigg\{  
   \sum_{j=1}^2i(A_j^{1/2}w,y_j)^{\sim}\bigg\}\\
&\qquad\times
\bigg(\prod_{j=1}^2E_{x_j} \Big[\exp\Big\{i\big(A_j^{1/2}w,
    \lambda_j^{-1/2}x_j +A_j^{1/2}g_j\big)^{\sim}\Big\}\Big]\bigg)d f(w)\\
&=\int_{C_{a,b}'[0,T]}
\exp\bigg\{  \sum_{j=1}^2i(A_j^{1/2}w,y_j)^{\sim}\bigg\}
\bigg(\prod_{j=1}^2E_{x_j} \big[G_j(w;x_j+x_{0,j})\big]\bigg)d f(w)\\
&=\int_{C_{a,b}'[0,T]}\exp \bigg\{  \sum_{j=1}^2i(A_j^{1/2}w,y_j)^{\sim}\bigg\} \\
&\qquad\times
\exp\bigg\{\sum_{j=1}^2\bigg[-\frac{\lambda_j}{2}(A_jg_{j},g_j)_{C_{a,b}'} 
-\lambda_j^{1/2}(A_j^{1/2}g_{j},a)_{C_{a,b}'}\bigg]\bigg\}\\
&\qquad\times
\bigg(\prod_{j=1}^2E_{x_j}\Big[\exp\Big\{i(A_j^{1/2}w,\lambda_j^{-1/2}x_j)^{\sim}
+\lambda_j^{1/2}( A_j^{1/2}g_j, x_j)^{\sim}\Big\}\Big]\bigg)df(w)\\
&=\exp\bigg\{\sum_{j=1}^2\bigg[-\frac{\lambda_j}{2}(A_jg_{j},g_j)_{C_{a,b}'} 
-\lambda_j^{1/2}(A_j^{1/2}g_{j},a)_{C_{a,b}'}\bigg]\bigg\}\\
&\quad\times
E_{\vec x}\bigg[\int_{C_{a,b}'[0,T]}\exp \bigg\{  
\sum_{j=1}^2\Big[i(A_j^{1/2}w,y_j)^{\sim} 
+i(A_j^{1/2}w,\lambda_j^{-1/2}x_j)^{\sim}\Big]\bigg\}df(w) \\
&\qquad\qquad\qquad\quad\times
\exp\bigg\{\sum_{j=1}^2\lambda_j^{1/2}
\big( A_j^{1/2}g_j,x_j\big)^{\sim}\bigg\} \bigg]\\
&=\exp\bigg\{\sum_{j=1}^2\bigg[-\frac{\lambda_j}{2}(A_jg_{j},g_j)_{C_{a,b}'} 
-\lambda_j^{1/2}(A_j^{1/2}g_{j},a)_{C_{a,b}'}
-\lambda_j(A_j^{1/2}g_j,y_j)^{\sim}\bigg]\bigg\}\\
&\quad\times
E_{\vec x} \bigg[F(y_1+\lambda_1^{-1/2}x_1,y_2+\lambda_2^{-1/2}x_2) \\
&\qquad\qquad\times
\exp\bigg\{\sum_{j=1}^2\bigg[\lambda_j(A_j^{1/2}g_j,y_j)^{\sim}
+\lambda_j^{1/2}\big( A_j^{1/2}g_j,x_j\big)^{\sim}
\bigg]\bigg\}\bigg]\\
\end{split}
\]
\[
\begin{split}
&= \exp\bigg\{\sum_{j=1}^2\bigg[
-\frac{\lambda_j}{2}(A_jg_{j},g_j)_{C_{a,b}'} 
-\lambda_j^{1/2}(A_j^{1/2}g_{j},a)_{C_{a,b}'}
-\lambda_j(A_j^{1/2}g_j,y_j)^{\sim}\bigg]\bigg\}\\
&\quad\times
E_{\vec x} \big[\Phi_1(y_1+\lambda_1^{-1/2}x_1,
    y_2+\lambda_2^{-1/2}x_2)\big],
\end{split}
\]
where
\[
\Phi_1(x_1,x_2) 
 =F( x_1,x_2) 
\exp\bigg\{\sum_{j=1}^2\bigg[\lambda_j(A_j^{1/2}g_j, x_j\big)^{\sim}
\bigg]\bigg\}.
\]

\par
On the other hand for $\lambda_1>0$ and $\lambda_2>0$, 
we see  that  
\[
\begin{split}
&T_{\vec \lambda} (F^*)(y_1,y_2)\\
&=E_{\vec x} \big[F^*\big(y_1+\lambda_1^{-1/2}x_1,
      y_2 +\lambda_2^{-1/2}x_2\big)\big]\\
&=E_{\vec x} \bigg[F\big(y_1+\lambda_1^{-1/2}x_1,
      y_2+\lambda_2^{-1/2}x_2\big)\\
&\qquad\quad\times
\exp\bigg\{   \sum\limits_{j=1}^2\Big[-iq_j(A_j^{1/2}g_j,y_j
      +\lambda_j^{-1/2}x_j)^{\sim}\Big]\bigg\}\bigg]\\
&=E_{\vec x} \bigg[F\big(y_1+\lambda_1^{-1/2}x_1,y_2
     +\lambda_2^{-1/2}x_2\big)\\
&\qquad\quad\times
\exp\bigg\{  \sum\limits_{j=1}^2
 \Big[-iq_j(A_j^{1/2}g_j,y_j)^{\sim}
-iq\lambda_j^{-1/2}(A_j^{1/2}g_j,x_j)^{\sim}\Big]\bigg\}\bigg].
\end{split}
\]

\par
Next,  using   H\"older's inequality 
with $\lambda_1>0$ and $\lambda_2>0$, 
it follows that
\[
\begin{split}
&E_{\vec x} \Big[\Big|F^*(y_1+\lambda_1^{-1/2}x_1,
     y_2+\lambda_2^{-1/2}x_2)\\
& \qquad \qquad
-\Phi_1(y_1+\lambda_1^{-1/2}x_1,
     y_2+\lambda_2^{-1/2}x_2)\Big|\Big]\\
&=E_{\vec x} \bigg[\bigg|F\big(y_1+\lambda_1^{-1/2}x_1,
     y_2+\lambda_2^{-1/2}x_2\big)\bigg|\\
&\qquad\times
\bigg|1-\exp\bigg\{  \sum\limits_{j=1}^2
    \Big[(iq_j+\lambda_j)(A_j^{1/2}g_j,y_j)^{\sim}\\
&\qquad\qquad\qquad\qquad\qquad
+(iq_j\lambda_j^{-1/2}+\lambda_j^{1/2})(A_j^{1/2}g_j,x_j)^{\sim}\Big]\bigg\}
 \bigg|\bigg]\\
&\le \bigg(E_{\vec x} \bigg[\bigg|F\big(y_1+\lambda_1^{-1/2}x_1,
     y_2+\lambda_2^{-1/2}x_2\big)\bigg|^p\bigg]\bigg)^{1/p}\\
&\quad\times
\bigg(E_{\vec x} \bigg[\bigg|1-\exp\bigg\{  \sum\limits_{j=1}^2
    \Big[(iq_j+\lambda_j)(A_j^{1/2}g_j,y_j)^{\sim}\\
&\qquad\qquad\qquad\qquad\qquad
+(iq_j\lambda_j^{-1/2}+\lambda_j^{1/2})(A_j^{1/2}g_j,x_j)^{\sim}\Big]\bigg\}
 \bigg|^{p'}\bigg]\bigg)^{1/p'}.
\end{split}
\]
Note that each factor in the last expression has 
a  limit as $\vec \lambda=(\lambda_1,\lambda_2) 
   \to-i\vec q =(-iq_1,-iq_2)$ in $\mathbb C_+^2$,
and that
\[
\begin{split}
&\bigg(E_{\vec x}\bigg[\bigg|1 
 -\exp\bigg\{  \sum\limits_{j=1}^2\Big[(iq_j+\lambda_j)(A_j^{1/2}g_j,y_j)^{\sim}\\
&\qquad\qquad
+(iq_j\lambda_j^{-1/2}+\lambda_j^{1/2})(A_j^{1/2}g_j,x_j)^{\sim}\Big]\bigg\}
 \bigg|^{p'}\bigg]\bigg)^{1/p'}\rightarrow 0\\
\end{split}
\]
as $\vec \lambda=(\lambda_1,\lambda_2) 
  \to-i\vec q= (-iq_1,-iq_2)$ in $\mathbb C_+^2$. 
Hence we have that
\[
\begin{split}
&T_{\vec q}^{(p)}(F)(y_1+A_1^{1/2}g_1,y_2+A_2^{1/2}g_2)\\
&=\operatorname*{l.i.m.}\limits_{\vec\lambda\to -i\vec q}
T_{\vec\lambda}(F)(y_1+A_1^{1/2}g_1,y_2+A_2^{1/2}g_2)\\
&=\operatorname*{l.i.m.}\limits_{\vec\lambda\to -i\vec q}
\exp\bigg\{\sum_{j=1}^2\bigg[-\frac{\lambda_j}{2}(A_jg_{j},g_j)_{C_{a,b}'} 
-\lambda_j^{1/2}(A_j^{1/2}g_{j},a)_{C_{a,b}'}\\
& \qquad\qquad\qquad\qquad
  -\lambda_j(A_j^{1/2}g_j,y_j)^{\sim}\bigg]\bigg\}
E_{\vec x}\big[\Phi_1(y_1+\lambda_1^{-1/2}x_1,y_2+\lambda_2^{-1/2}x_2)\big]\\
&=\exp\bigg\{\sum_{j=1}^2\bigg[ \frac{iq_j}{2}(A_jg_{j},g_j)_{C_{a,b}'} 
-(-iq)^{1/2}(A_j^{1/2}g_{j},a)_{C_{a,b}'} 
  +iq_j(A_j^{1/2}g_j,y_j)^{\sim}\bigg]\bigg\}\\
& \qquad \times
\operatorname*{l.i.m.}\limits_{\vec\lambda\to -i\vec q}
E_{\vec x}\big[\Phi_1(y_1+\lambda_1^{-1/2}x_1,y_2+\lambda_2^{-1/2}x_2)\big]\\
&=\exp\bigg\{\sum_{j=1}^2\bigg[ \frac{iq_j}{2}(A_jg_{j},g_j)_{C_{a,b}'} 
-(-iq)^{1/2}(A_j^{1/2}g_{j},a)_{C_{a,b}'}
  +iq_j(A_j^{1/2}g_j,y_j)^{\sim}\bigg]\bigg\}\\
& \qquad \times
\operatorname*{l.i.m.}\limits_{\vec\lambda\to -i\vec q}
T_{\vec \lambda}(F^*)(y_1 ,y_2 )\\
&=\exp\bigg\{\sum_{j=1}^2\bigg[ \frac{iq_j}{2}(A_jg_{j},g_j)_{C_{a,b}'} 
-(-iq)^{1/2}(A_j^{1/2}g_{j},a)_{C_{a,b}'}
  +iq_j(A_j^{1/2}g_j,y_j)^{\sim}\bigg]\bigg\}\\
& \qquad \times
T_{\vec q}(F^*)(y_1 ,y_2 ).\\
\end{split}
\]
\end{proof}

\par
The following corollary follows from 
equation   \eqref{eq:t1q0} above.

\begin{corollary}\label{eq:coro010}
Let $q_0$, $F$, $g_1$ and $g_2$ be 
as in Theorem \ref{thm:tpq-trans}.
Then for all real numbers  $q_1$ and $q_2$ 
with $|q_j|> q_0$, $j\in\{1,2\}$,
\[
\begin{split}
&E_{\vec x}^{\text{\rm anf}_{\vec q}}
[F(x_1+A_1^{1/2}g_1,x_2+A_2^{1/2}g_2)  ]\\
&=\exp\bigg\{\sum\limits_{j=1}^2
\bigg[\frac{iq_j}{2}(A_jg_j,g_j)_{C_{a,b}'}
    -(-iq_j)^{1/2} (A_j^{1/2}g_j,a)_{C_{a,b}'}\bigg]\bigg\}\\
&\quad \times
  E_{\vec x}^{\text{\rm anf}_{\vec q}}
 \bigg[ F(x_1,x_2)
\exp\bigg\{   \sum\limits_{j=1}^2
\Big[-iq(A_j^{1/2}g_j,x_j)^{\sim}\Big]\bigg\} \bigg].
\end{split}     
\] 
\end{corollary}

\par
By (1) of Remark \ref{re:A} and Corollary  \ref{eq:coro010} above, 
we have the following corollary.

\begin{corollary}\label{coro:cab}
Let $q$ be a nonzero real number, 
let $g_0\in C_{a,b}'[0,T]$ and 
let  $F$ be an  element of   $\mathcal F(C_{a,b}[0,T])$ 
given by equation \eqref{eq:element}. 
Then
\begin{equation}\label{eq:cab}
\begin{split}
&\int_{C_{a,b}[0,T]}^{\text{\rm  anf}_{q}}
    F( x  + g_0 ) d  \mu(x)\\
&\stackrel{*}{=}
\exp\bigg\{\frac{iq}{2}  \|g_0\|_{C_{a,b}'}^{2}
    -(-iq)^{1/2} ( g_0,a)_{C_{a,b}'}\bigg\} \\
&\quad \times
  \int_{C_{a,b}[0,T]}^{\text{\rm  anf}_{q}} 
   F(x )\exp \{- iq ( g_0,x)^{\sim} \} d \mu( x)\\
\end{split}
\end{equation}
where  $\stackrel{*}{=}$   means that if either side exists,
then both sides exist and   equality holds.
\end{corollary}

\begin{remark}
In Corollary \ref{coro:cab},
taking $a(t)\equiv 0$ and $b(t)=t$,     
the general function space $C_{a,b}[0,T]$
reduces to the classical Wiener space $C_0[0,T]$.
Also, we know that 
equation \eqref{eq:cab} becomes
\begin{equation*}\label{eq:wiener}
\begin{split}
&\int_{C_0[0,T]}^{\text{\rm  anf}_{q}}F( x  + g_0 ) d  m_w (x)\\
&=\exp\bigg\{\frac{iq}{2}  \|g_0'\|_2^{2}\bigg\}
  \int_{C_0[0,T] }^{\text{\rm  anf}_q} F(x )
\exp \bigg\{- iq \int_0^Tg_0'(t)d x(t)\bigg\} d m_w ( x)
\end{split}
\end{equation*}
where $\|\cdot\|_2$ is the norm on $L^2[0,T]$.
This result subsume similar known result
obtained by Cameron and Storvick \cite{CS82}.
\end{remark}

\setcounter{equation}{0}
\section{A Cameron-Storvick type theorem on $C_{a,b}^2[0,T]$}\label{sec:8}

In \cite{Cameron51}, 
Cameron (see \cite[Theorem A, p.145]{CS91})
expressed  the Wiener integral 
of the first variation of a functional $F$ 
in terms of the Wiener integral 
of the product of $F$ by a linear functional, 
and in \cite[Theorem 1]{CS91},
Cameron and Storvick obtained  
a similar result for the analytic Feynman integral 
on classical Wiener space.
In   \cite[Theorem 2.4, p.491]{CSY01},  
Chang, Song and Yoo also obtained
a Cameron-Storvick  theorem  on abstract Wiener space.
In \cite{CS03}, 
Chang and Skoug obtained  these results
for  functionals on the function space $C_{a,b}[0,T]$.
Also see \cite{CS02, CCS04} for related results involving
conditional (generalized) Feynman integrals  
and Fourier-Feynman transforms.

\par
In order to establish similar results for functionals 
in $\mathcal{F}_{A_1,A_2}^{\,\,a,b}$
(see Theorem \ref{thm:CS-B2} below) 
we first obtain a Cameron-Storvick type theorem
for the function space $C_{a,b}^2[0,T]$.

\begin{theorem}\label{thm:CS}
Let $g_1$ and $g_2$ be elements of $C_{a,b}'[0,T]$.
Let $F(x_1,x_2)$ be $\mu\times \mu$-integrable over $C_{a,b}^2[0,T]$. 
Assume that   $F$ has  a    first variation
$\delta F (x_1,x_2|g_1,g_2)$ for all $(x_1,x_2)\in C_{a,b}^2[0,T]$  
such that    for some  $\gamma>0$,
\[
\sup_{|h|\le \gamma }
\big|\delta F   (x_1 +h  g_1,x_2 +h  g_2|  g_1,g_2)\big|
\]
is $\mu\times \mu$-integrable over $C_{a,b}^2[0,T]$
as a function  of $(x_1,x_2)\in C_{a,b}^2[0,T]$.  
Then 
\begin{equation}\label{eq:CA-basic} 
\begin{split}
&E_{\vec x}[ \delta F(x_1,x_2|g_1,g_2)]\\
&  \quad  =
 E_{\vec x}\big[ F(x_1, x_2)\big\{(g_1, x_1)^{\sim}
   +(g_2, x_2)^{\sim}\big\}\big]\\
& \qquad 
-\big\{(g_1, a)_{C_{a,b}'}+(g_2, a)_{C_{a,b}'}\big\}
E_{\vec x}[ F(x_1,x_2)  ].
\end{split}
\end{equation}
\end{theorem}
\begin{proof}
First note that 
\[
\begin{split}
&\delta F (x_1+h  g_1,x_2+h g_2 |  g_1,g_2)\\
& =\frac{\partial}{\partial \lambda}
   F (x_1 +h g_1+\lambda g_1,x_2+h g_2  )\bigg|_{\lambda=0}
+\frac{\partial}{\partial \lambda}
   F (x_1+h g_1, x_2 +h g_2+\lambda g_2  )\bigg|_{\lambda=0}\\
& =\frac{\partial}{\partial \lambda}
   F (x_1 +(h  +\lambda) g_1,x_2+h g_2  )\bigg|_{\lambda=0}
+\frac{\partial}{\partial \lambda}
   F (x_1+h g_1, x_2 +(h +\lambda) g_2  )\bigg|_{\lambda=0}\\
& =\frac{\partial}{\partial \mu}
   F (x_1 + \mu g_1,x_2+h g_2  )\bigg|_{\mu=h}
+\frac{\partial}{\partial \mu}
   F (x_1+h g_1, x_2 +\mu g_2  )\bigg|_{\mu=h}\\
& =2 \frac{\partial}{\partial h}
   F (x_1 + h g_1,x_2+h g_2  ).
\end{split}
\]
But since 
\[
\sup_{|h|\le \gamma }
 \bigg| \frac{\partial}{\partial h}
   F (x_1 + h g_1,x_2+h g_2  )\bigg|
\]
is $\mu\times\mu$-integrable,
\[
\frac{\partial}{\partial h}   F (x_1 + h g_1,x_2+h g_2  )
\]
is $\mu\times\mu$-integrable  for sufficiently small values of $h$.
Hence by the Fubini theorem  and equation \eqref{eq:translation},
we see that 
\[
\begin{split}
&E_{\vec x}[ \delta   F(x_1, x_2|g_1,g_2)]\\
& = E_{\vec x}\bigg[ \frac{\partial}{\partial h}  F(x_1+h g_1, x_2) \Big|_{h=0}\bigg]
+   E_{\vec x}\bigg[\frac{\partial}{\partial h} F(x_1, x_2 +h g_2 ) \Big|_{h=0}\bigg]\\
& =  E_{x_2}\bigg[\frac{\partial}{\partial h}
   E_{x_1}[ F(x_1+h g_1, x_2)] \Big|_{h=0} \bigg]  
  +E_{x_1} \bigg[ \frac{\partial}{\partial h} 
     E_{x_2}[F(x_1, x_2 +h g_2 )  ]\Big|_{h=0} \bigg] \\
& =E_{x_2}\bigg[\frac{\partial}{\partial h}\bigg(
   \exp\bigg\{-\frac{h^2}{2}\|g_1\|_{C_{a,b}'}^2-h(g_1,a)_{C_{a,b}'}\bigg\}\\
& \qquad\qquad\qquad\qquad\times
E_{x_1}\big[ F(x_1, x_2)\exp\{h(g_1,x_1)^{\sim}\}\big]\bigg) 
\bigg|_{h=0} \bigg] \\
& \quad +
  E_{x_1} \bigg[  \frac{\partial}{\partial h } \bigg(
       \exp\bigg\{-\frac{h^2}{2}\|g_2\|_{C_{a,b}'}^2-h(g_2,a)_{C_{a,b}'}\bigg\}\\
& \qquad\qquad\qquad\qquad\times
   E_{x_2}\big[F(x_1, x_2)\exp\{h(g_2,x_2)^{\sim}\} \big] 
\bigg)\bigg|_{h=0} \bigg] \\
& = E_{\vec x}\big[ F(x_1, x_2)\big\{(g_1, x_1)^{\sim}+(g_2, x_2)^{\sim}\big\}\big]\\
& \qquad 
-\big\{(g_1, a)_{C_{a,b}'}+(g_2, a)_{C_{a,b}'}\big\}
E_{\vec x}\big[ F(x_1,x_2)  \big].
\end{split}
\]
\end{proof}

\begin{lemma}\label{lemma:CS}
Let $g_1$, $g_2$,  and $F$
be as in Theorem \ref{thm:CS}.
For each $\rho_1>0$ and $\rho_2 >0$, assume that   
$F(\rho_1 x_1,\rho_2 x_2)$ is $\mu\times\mu$-integrable.
Furthermore assume that $F (\rho_1 x_1, \rho_2 x_2)$ 
has a  first variation
$\delta F(\rho_1 x_1, \rho_2 x_2|\rho_1 g_1,\rho_2 g_2)$ 
for all $(x_1,x_2) \in  C_{a,b}^2[0,T]$  
such that    for some  positive function $\gamma(\rho_1,\rho_2)$,
\[
\sup_{|h|\le \gamma(\rho_1,\rho_2)}
\big|\delta F   (\rho_1 x_1 +\rho_1 h  g_1,\rho_2 x_2 +\rho_2 h  g_2|
      \rho_1 g_1,\rho_2 g_2)\big|
\]
is   $\mu\times \mu$-integrable over $C_{a,b}^2[0,T]$
as a function  of $(x_1,x_2)\in C_{a,b}^2[0,T]$.  
Then 
\begin{equation} \label{eq:CA-b2}
\begin{split}
&E_{\vec x} \big[\delta F(\rho_1 x_1,\rho_2 x_2|\rho_1 g_1,\rho_2 g_2) \big]\\
&  \quad  =
  E_{\vec x}\big[ F(\rho_1 x_1,\rho_2  x_2)  
\big\{(g_1,x_1)^{\sim}+(g_2, x_2)^{\sim}\big\} \big]\\
&\qquad 
-\big\{(g_1,a)_{C_{a,b}'}+(g_2, a)_{C_{a,b}'}\big\}
 E_{\vec x}\big[F(\rho_1 x_1,\rho_2  x_2)   \big].
\end{split}
\end{equation}
\end{lemma}
\begin{proof}
Let $R(x_1,x_2)=F(\rho_1 x_1,\rho_2 x_2)$.
Then we have that
\[
R(x_1+ h g_1, x_2)= F(\rho_1 x_1 + \rho_1 h g_1, \rho_2 x_2)
\]
and 
\[
R(x_1, x_2+ h g_2)= F(\rho_1 x_1, \rho_2 x_2 + \rho_2 h g_2)
\]
and that
\[
\frac{\partial }{\partial h}R(x_1+ h g_1, x_2)\bigg|_{h=0}
= \frac{\partial }{\partial h}F(\rho_1 x_1 + \rho_1 h g_1, \rho_2 x_2)\bigg|_{h=0}
\]
and 
\[
\frac{\partial }{\partial h}R(x_1, x_2+ h g_2)\bigg|_{h=0}
= \frac{\partial }{\partial h}F(\rho_1 x_1, 
  \rho_2 x_2 + \rho_2 h g_2)\bigg|_{h=0}.
\]
Thus we have
\[
\begin{split}
&\delta F(\rho_1 x_1,\rho_2 x_2|\rho_1 g_1, \rho_2 g_2)\\
&= \frac{\partial }{\partial h}F(\rho_1 x_1 + \rho_1 h g_1, \rho_2 x_2)\bigg|_{h=0}
   +\frac{\partial }{\partial h}F(\rho_1 x_1  , \rho_2x_2+ \rho_2h g_2)\bigg|_{h=0}\\
&= \frac{\partial }{\partial h}R(x_1+ h w_1, x_2)\bigg|_{h=0}
   +\frac{\partial }{\partial h}R(x_1, x_2+ h g_2)\bigg|_{h=0} \\
&=\delta R(  x_1, x_2| g_1, g_2).
\end{split}
\]
Hence by equation \eqref{eq:CA-basic}, we have
\[
\begin{split}
&E_{\vec x} \big[ \delta F(\rho_1 x_1,\rho_2 x_2|\rho_1g_1,\rho_2g_2) \big]\\
&=E_{\vec x} \big[ \delta R( x_1, x_2| g_1, g_2) \big]\\
&=E_{\vec x} \big[ R(x_1, x_2)\big\{(g_1,x_1)^{\sim}+(g_2, x_2)^{\sim}\big\}\big]\\
& \qquad 
-\big\{(g_1,a)_{C_{a,b}'}+(g_2, a)_{C_{a,b}'}\big\}
 E_{\vec x} \big[ R(x_1,x_2) \big]\\
&=E_{\vec x} \big[ F(\rho_1 x_1,\rho_2  x_2)\big\{(g_1,x_1)^{\sim}+(g_2, x_2)^{\sim}\big\}\big]\\
& \qquad 
-\big\{(g_1,a)_{C_{a,b}'}+(g_2, a)_{C_{a,b}'}\big\}E_{\vec x} \big[F(\rho_1x_1,\rho_2x_2) \big]
\end{split}
\]
which establishes equation \eqref{eq:CA-b2}.
\end{proof}

\begin{theorem} \label{thm:CS-B2} 
Let $g_1$, $g_2$,   and $F$
be as in Lemma \ref{lemma:CS}.
Then if any two of the three generalized  analytic Feynman integrals 
in the following equation exist,
then the third one also exists, and equality holds:
\begin{equation}\label{eq:CSthm}
\begin{split}
&E_{\vec x}^{\text{\rm anf}_{(q_1,q_2)}}\big[\delta  F( x_1, x_2|g_1, g_2)\big]\\
&=
-iE_{\vec x}^{\text{\rm anf}_{(q_1,q_2)}} \big[F( x_1, x_2)\big\{q_1(g_1,x_1)^{\sim}
+q_2 (g_2,x_2)^{\sim}\big\}\big]\\
&\quad
-i\Big\{(-iq_1)^{1/2}(g_1,a)_{C_{a,b}'}
  +(-iq_2)^{1/2}(g_2,a)_{C_{a,b}'}\Big\}
E_{\vec x}^{\text{\rm anf}_{(q_1,q_2)}}\big[F( x_1, x_2)\big].
\end{split}
\end{equation} 
\end{theorem}
\begin{proof}
Let $\rho_1>0$ and $\rho_2>0$ be given.
Let $y_1= \rho_1^{-1}g_1$ and $y_2= \rho_2^{-1}g_2$.
By equation \eqref{eq:CA-b2},
\begin{equation}\label{eq:CS-l}
\begin{split}
&E_{\vec x} \big[\delta F(\rho_1 x_1,\rho_2 x_2| g_1,  g_2) \big]\\
&=E_{\vec x} \big[ \delta F(\rho_1 x_1,\rho_2 x_2| 
  \rho_1 y_1, \rho_2 y_2) \big]\\
&= E_{\vec x} \big[ F(\rho_1 x_1,\rho_2  x_2) 
  \big\{(y_1,x_1)^{\sim}+(y_2, x_2)^{\sim}\big\}\big]\\
&\quad 
-\big\{(y_1,a)_{C_{a,b}'}+(y_2, a)_{C_{a,b}'}\big\}
   E_{\vec x} \big[ F(\rho_1 x_1,\rho_2  x_2)  \big]\\
&= E_{\vec x} \big[ F(\rho_1 x_1,\rho_2  x_2) 
\big\{\rho_1^{-2}(g_1,\rho_1x_1)^{\sim}
  +\rho_2^{-2}(g_2, \rho_2x_2)^{\sim}\big\}\big]\\
&\quad 
-\big\{\rho_1^{-1}(g_1,a)_{C_{a,b}'}+\rho_2^{-1}(g_2, a)_{C_{a,b}'}\big\}
E_{\vec x} \big[ F(\rho_1 x_1,\rho_2  x_2) \big].
\end{split}
\end{equation}
Now let $\rho_1=\lambda_1^{-1/2}$ and $\rho_2=\lambda_2^{-1/2}$. 
Then equation \eqref{eq:CS-l} becomes
\begin{equation}\label{eq:CS-ll}
\begin{split}
&E_{\vec x} \big[ \delta F(\lambda_1^{-1/2} x_1,
   \lambda_2^{-1/2}x_2| g_1,  g_2) \big]\\
&= E_{\vec x} \big[ F(\lambda_1^{-1/2} x_1,\lambda_2^{-1/2}  x_2)   \
\big\{\lambda_1 (g_1,\lambda_1^{-1/2}x_1)^{\sim}
    +\lambda_2 (g_2, \lambda_2^{-1/2}x_2)^{\sim}\big\}  \big]\\
&\quad 
-\big\{\lambda_1^{1/2}(g_1,a)_{C_{a,b}'}
    +\lambda_2^{1/2}(g_2, a)_{C_{a,b}'}\big\}
E_{\vec x} \big[ F(\lambda_1^{-1/2} x_1,\lambda_1^{-1/2}  x_2)  \big].\\
\end{split}
\end{equation}
Since $\rho_1>0$ and $\rho_2>0$ were arbitrary,
we have that equation \eqref{eq:CS-ll} 
holds for all $\lambda_1>0$ and $\lambda_2>0$.
We now use  Definition \ref{def:gfi} to obtain
our desired conclusions.
\end{proof}

\setcounter{equation}{0}
\section{Applications of the Cameron-Storvick type theorem}\label{sec:last}

In this section we consider functionals 
in  the generalized Fresnel type class 
$\mathcal{F}_{A}^{\,\,a,b}\equiv \mathcal{F}_{A_+,A_-}^{\,\,a,b}$
where $A$, $A_+$ and $A_-$ are related by 
equation \eqref{eq:decomposition} above.

\par
Let $\phi$ be a function of bounded variation on $[0,T]$.
Define an operator $A:C_{a,b}'[0,T]\to C_{a,b}'[0,T]$ by
\[
Aw(t)
=\int_0^t \phi(s)D_swd b(s)
=\int_0^t \phi(s)\frac{w'(s)}{b'(s)} d b(s)
=\int_0^t \phi(s)z(s)d b(s)
\]
for $w(t)=\int_0^tz(s)d b(s)$.
It is easily shown that  $A$ is a  self-adjoint operator.
We also see that $A=A_+-A_-$ where
\[
A_+w(t)
=\int_0^t \phi^{+}(s)D_swd b(s)
\,\,\hbox{ and }\,\,
A_-w(t)
=\int_0^t \phi^{-}(s)D_swd b(s)
\]
and $\phi^+$ and $\phi^-$ are the positive part 
and the negative part of $\phi$, respectively.
Also, $A_+^{1/2}$ and $A_-^{1/2}$ are given by
\[
A_+^{1/2}w(t)
=\int_0^t \sqrt{\phi^{+}}(s)D_swd b(s)
\,\,\hbox{ and }\,\,
A_-^{1/2}w(t)
=\int_0^t \sqrt{\phi^{-}}(s)D_swd b(s),
\]
respectively. 
For a more   detailed study of this decomposition, 
see \cite[pp.187-189]{JL00}.

\par
For fixed $g\in C_{a,b}'[0,T]$, 
let $g_1=A_+^{1/2}g$ and $g_2=A_-^{1/2}(-g)$.
Then we see that for all  $w(t)=\int_0^tz(s)d b(s)$ in $C_{a,b}'[0,T]$,
\begin{equation}\label{eq:q001}
\begin{split}
&(A_+^{1/2}w, g_1)_{C_{a,b}'}+(A_-^{1/2}w, g_2)_{C_{a,b}'}\\
&=(A_+^{1/2}w, A_+^{1/2}g)_{C_{a,b}'}
 -(A_-^{1/2}w,A_-^{1/2} g)_{C_{a,b}'}\\
&=(A_+w, g)_{C_{a,b}'}-(A_- w, g)_{C_{a,b}'}\\
&=(Aw,g)_{C_{a,b}'}.
\end{split}
\end{equation}
Using equation \eqref{eq:Dt}, we also  see that
\begin{equation}\label{eq:q002}
\begin{split}
(A_+^{1/2} w, a)_{C_{a,b}'}
&=(w,A_+^{1/2} a)_{C_{a,b}'}\\
&=\int_0^TD_t w\sqrt{\phi^{+}}(t)\frac{a'(t)}{b'(t)}db(t) \\
&=\int_0^TD_tw\sqrt{\phi^{+}}(t)a'(t)dt\\
&\equiv (D_tw\sqrt{\phi^+},a')
\end{split}
\end{equation}
and
\begin{equation}\label{eq:q003}
(A_-^{1/2} w, a)_{C_{a,b}'}
=\int_0^TD_tw\sqrt{\phi^{-}}(t)a'(t)dt\equiv (D_tw\sqrt{\phi^-},a').
\end{equation}

\par
Assume that  $F$ is an element of 
$\mathcal{F}_{A}^{\,q_0}\cap\mathcal{G}_{A}^{\,q_0}$
for some $q_0\in(0,1)$ 
where $\mathcal{F}_{A}^{\,q_0}\equiv \mathcal{F}_{A_+,A_-}^{\,\,q_0}$
and $\mathcal{G}_{A}^{\,q_0}\equiv \mathcal{G}_{A_+,A_-}^{\,\,q_0}$
(the classes  $\mathcal{F}_{A_1,A_2}^{\,\, q_0}$ and 
$\mathcal{G}_{A_1,A_2}^{\,\, q_0}$  are defined 
in Sections \ref{sec:3} and \ref{sec:5}, respectively).

\par
Using   \eqref{eq:gfi-delta} with $(q_1, q_2)=(1,-1)$, \eqref{eq:q001},
\eqref{eq:q002} and \eqref{eq:q003},   
we  obtain that
\begin{equation}\label{eq:connection}
\begin{split}
&E_{\vec x}^{\text{\rm{anf}}_{(1,-1)}}\big[
\delta F(x_1,x_2|A_+^{1/2}g  , -A_-^{1/2}g  )\big]\\
&=E_{\vec x}^{\text{\rm{anf}}_{(1,-1)}}\big[
  \delta F(x_1,x_2|g_1  , g_2  )\big]\\
& =\int_{C_{a,b}'[0,T]}  \Big[i(A_+ ^{1/2}w, A_+^{1/2} g)_{C_{a,b}'}
              -i(A_- ^{1/2}w, A_- ^{1/2} g)_{C_{a,b}'} \Big]\\
& \qquad\quad \times 
\exp\bigg\{   -\frac{i}{2}((A_+-A_-)w,w)_{C_{a,b}'}\bigg\}\\
& \quad\times
\exp\bigg\{ i\Big[(-i)^{-1/2}(A_+^{1/2}w,a)_{C_{a,b}'}
   +(i)^{-1/2}(A_-^{1/2}w,a)_{C_{a,b}'}\Big] \bigg\} df(w)\\
& =\int_{C_{a,b}'[0,T]} i(Aw,  g)_{C_{a,b}'}
\exp\bigg\{   -\frac{i}{2}(Aw,w)_{C_{a,b}'}\bigg\}\\
&\quad\times
\exp\bigg\{i\Big[(-i)^{-1/2} (D_tw\sqrt{\phi^+},a')
    +(i)^{-1/2} (D_tw\sqrt{\phi^-},a')\Big] \bigg\} df(w).\\
\end{split}
\end{equation}
We also see that for all $\rho_1>0$, $\rho_2>0$ and $h\in \mathbb R$
\[
\begin{split}
&\big|\delta F(\rho_1x_1+\rho_1 h g_1,  
     \rho_2x_2+\rho_2h g_2| \rho_1g_1,\rho_2g_2 )\big|\\
& \le 
\int_{C_{a,b}'[0,T]} \Big[\big|(A_+ ^{1/2}w, 
       A_+^{1/2}(\rho_1 g))_{C_{a,b}'} \big|
   +\big|(A_- ^{1/2}w, A_- ^{1/2}(-\rho_1 g))_{C_{a,b}'} \big| \Big] 
d |f| (w) \\
& \le 
\rho_1  \int_{C_{a,b}'[0,T]} \big|(A_+ w,  g )_{C_{a,b}'} \big| 
d |f| (w)
   + \rho_2 \int_{C_{a,b}'[0,T]} \big|(A_- w,  g)_{C_{a,b}'}  \big| 
d |f| (w)\\
& \le 
\big(\rho_1 \|A_+\|_o +\rho_2 \|A_-\|_o\big) \|g\|_{C_{a,b}'}
\int_{C_{a,b}'[0,T]}\|w \|_{C_{a,b}'}d |f| (w).
\end{split}
\]
But the last expression above
is bounded and is independent of $(x_1,x_2)\in C_{a,b}^2[0,T]$.
Hence 
$\delta F(\rho_1x_1+\rho_1 h g_1,\rho_2x_2+\rho_2 h g_2
    | \rho_1 g_1,\rho_2 g_2 )$
is $\mu\times\mu$-integrable in 
$(x_1,x_2)\in C_{a,b}^2[0,T]$ for every $\rho_1>0$ and $\rho_2>0$.
Also by Theorem \ref{thm:limit-1st} and Corollary \ref{coro:t1q-feynman},
the generalized analytic Feynman integrals
$E_{\vec x}^{\text{\rm anf}_{(1,-1)}} 
   [\delta F(x_1,x_2|A_+^{1/2}g,-A_-^{1/2}g)]$
and
$E_{\vec x}^{\text{\rm anf}_{(1,-1)}}[\delta F(x_1,x_2)]$  exist.
Thus by equation \eqref{eq:CSthm} together with equation \eqref{eq:connection}
we have 
\[
\begin{split}
&\int_{C_{a,b}'[0,T]} i(Aw,  g)_{C_{a,b}'}
\exp\bigg\{   -\frac{i}{2}(Aw,w)_{C_{a,b}'}\bigg\}\\
&\qquad\quad\times
\exp\bigg\{i\Big[(-i)^{-1/2} (D_tw\sqrt{\phi^+},a') 
   +(i)^{-1/2} (D_tw\sqrt{\phi^-},a')\Big] \bigg\} df(w)\\
&=-iE_{\vec x}^{\text{\rm anf}_{(1,-1)}} 
   \big[F( x_1, x_2)\big\{(A_+^{1/2}g,x_1)^{\sim}
+ (A_-^{-1/2}g,x_2)^{\sim}\big\}\big]\\
&\quad
-i\Big\{(-i )^{1/2}(D_tg\sqrt{\phi^+},a')
   -(i)^{1/2}(D_tg\sqrt{\phi^-},a')\Big\}
E_{\vec x}^{\text{\rm anf}_{(1,-1)}}\big[ F( x_1, x_2) \big].
\end{split}
\]


\section*{Acknowledgments}
 Seung Jun Chang worked as a corresponding author of this paper.
 The present research was conducted 
by the research fund of Dankook University in 2011.





\end{document}